\documentclass[a4paper,10pt,reqno]{amsart}
\usepackage{amsmath}
\usepackage{cases}
\usepackage{amsfonts}
\usepackage[colorlinks,linkcolor=blue,citecolor=blue]{hyperref}
\usepackage{latexsym, amssymb, amsmath, amsthm, bbm}
\usepackage[all]{xy}
\usepackage{pgfplots}
\usepackage{enumerate}
\usepackage{mathrsfs}

\DeclareSymbolFont{EulerExtension}{U}{euex}{m}{n}
\DeclareMathSymbol{\euintop}{\mathop} {EulerExtension}{"52}
\DeclareMathSymbol{\euointop}{\mathop} {EulerExtension}{"48}

\allowdisplaybreaks[4]

\def \id{\operatorname{id}}

\def \k{\Bbbk}

\def \span{\operatorname{span}}

\usepackage{mathtools}

\numberwithin{equation}{section}

\newtheorem{theorem}{Theorem}[section]
\newtheorem{lemma}[theorem]{Lemma}
\newtheorem{proposition}[theorem]{Proposition}
\newtheorem{corollary}[theorem]{Corollary}
\newtheorem{definition}[theorem]{Definition}
\newtheorem{example}[theorem]{Example}
\newtheorem{remark}[theorem]{Remark}

\begin{document}
\title[Coquasitriangular structures of abelian extensions]{Coquasitriangular structures on Hopf algebras constructed via abelian extensions}
\thanks{}
\author[J. Yu]{Jing Yu$^\dag$}
\author[X. Zhen]{Xiangjun Zhen}
\address{School of Mathematical Sciences, University of Science and Technology of China, Hefei 230026, China}
\email{yujing46@ustc.edu.cn}
\address{School of Mathematics, Nanjing University, Nanjing 210093, China}
\email{xjzhen@smail.nju.edu.cn}


\thanks{$^\dag$ Corresponding author}
\date{}
\maketitle
\begin{abstract}
The aim of this paper is to study coquasitriangular structures on a class of cosemisimple Hopf algebras of the form $\Bbbk^G {}^\tau \#_{\sigma} \Bbbk F$, constructed as abelian extensions of $\Bbbk F$ by $\Bbbk^G$ for a finite group $G$ and an arbitrary group $F$. We investigate when a coquasitriangular structure exists on $\Bbbk^G{}^\tau\#_{\sigma}\Bbbk F$ and provide characterizations of its coquasitriangular structures.
As an application, we study the coquasitriangular structures for the case where $G = \mathbb{Z}_2$.
\end{abstract}

\section{Introduction}
Quasitriangular Hopf algebras, introduced by Drinfeld, play a crucial role in the theory of quantum groups \cite{Dri86}. The introduction of the dual concept, the coquasitriangular Hopf algebra, was credited to the independent works of Majid \cite{Maj91} and Larson with Towber \cite{LT91}. Since then, significant attention has been directed towards determining whether a Hopf algebra is (co)quasitriangular, as well as classifying all possible (co)quasitriangular structures on it. Some researches related to this topic can be found in \cite{Ago13, AEG01, GW98, HL10,Jia05,JW05,Kei18,Nat06,Nik19}.

Recall that an abelian extension of Hopf algebras is a sequence $ K\xrightarrow{\iota} H \xrightarrow{\pi} A $ with $A$ cocommutative and $K$ commutative. Such extensions have been particularly useful for classifying semisimple Hopf algebras. A key example is the abelian extension of $\k F$ by $\k^G$, where $F, G$ are finite groups. Such extensions were classified by Masuoka (\cite[Proposition 1.5]{Mas02}), who showed that the corresponding Hopf algebra $H$ can be expressed as $\k^G\#_{\sigma, \tau}\k F$. The question of existence of quasitriangular structures on
$\k^G\#_{\sigma, \tau}\k F$ has been considered before. In 2011, Natale \cite{Nat11} proved that bicrossed product Hopf algebras arising from exact factorizations in almost simple finite groups admit no quasitriangular structure. Zhou and his cooperators \cite{ZL21,ZZ24} studied the explicit quasitriangular structures on
$\k^G\#_{\sigma, \tau}\k F$ in the case where
$F=\mathbb{Z}_2$.

In \cite{YLZZ25}, the authors removed the finiteness condition on $F$ and constructed a class of infinite-dimensional cosemisimple Hopf algebras $\Bbbk^G{}^\tau\#_{\sigma}\Bbbk F$ through abelian extensions of $\k F$ by $\k^G$. A natural question that arises is whether the study of quasitriangular structures on
$\k^G\#_{\sigma, \tau}\k F$ for finite groups
$F$ can be extended to the case of coquasitriangular structures on $\Bbbk^G{}^\tau\#_{\sigma}\Bbbk F$ when
$F$ is an infinite group. Moreover, we seek to establish characterizations for the existence of coquasitriangular structures on such Hopf algebras in full generality, without imposing additional constraints.

In this paper, we aim to establish necessary conditions for the existence of a coquasitriangular structure on the Hopf algebra
$\Bbbk^G{}^\tau\#_{\sigma}\Bbbk F$, where $G$ is a finite group and $F$ is arbitary group, and to investigate the characterization of such structures. Our approach is based on the fact that the category of comodules over a coquasitriangular Hopf algebra forms a braided monoidal category. As a consequence, if
$\Bbbk^G{}^\tau\#_{\sigma}\Bbbk F$ is coquasitriangular, it follows that the Grothendieck ring of its comodule category is commutative.
Note that the category of finite-dimensional right comodules over $\Bbbk^G{}^\tau\#_{\sigma}\Bbbk F$ is a semisimple tensor category (see Lemma \ref{lem:cose}).
The authors in \cite{YLZZ25} showed that all simple right $\k^G{}^\tau\#_{\sigma}\k F$-comodules can be described as induced comodules from certain coalgebras $\k^{G_f}_{\tau_f}$ for any $f\in F$. Furthermore, they proved that the Grothendieck ring of the category of finite-dimensional comodules over a cosemisimple Hopf algebra is completely determined by the ring structure of its irreducible characters.

Based on these observations, we can investigate when the Grothendieck ring of the comodule category of
$\Bbbk^G{}^\tau\#_{\sigma}\Bbbk F$ is commutative, thereby obtaining necessary conditions for
$\Bbbk^G{}^\tau\#_{\sigma}\Bbbk F$ to admit a coquasitriangular structure. Our main results are Theorems \ref{prop:OfOf} and \ref{prop:cqtandcomod}, stating that:
\begin{theorem}
Suppose $\Bbbk^G{}^\tau\#_{\sigma}\Bbbk F$ admits a coquasitriangular structure. Then
$O_fO_{f^\prime}=O_{f^\prime}O_f$ for all $f, f^\prime\in F.$
\end{theorem}

\begin{theorem}
Let $\Bbbk^G{}^\tau\#_\sigma\Bbbk F$ be a coquasitriangular Hopf algebra. For any $f\in F$, if $V = \operatorname{span}\{v^{(1)}, \dots, v^{(m)}\}$ and $W = \operatorname{span}\{w^{(1)}, \dots, w^{(n)}\}$ are simple right comodules over $\Bbbk^{G_f}_{\tau_f}$ and $\k^G$, with coactions $$\rho(v^{(i)})=\sum_{l=1}^m v^{(l)}\otimes (\sum_{g\in G_f}a_{li}^g p_g),$$ and $$\rho(w^{(j)})=\sum_{k=1}^n w^{(k)}\otimes (\sum_{h\in G}b_{kj}^{h} p_{h}),$$
defined for all $1\leq l, i\leq  m, 1\leq k, j\leq n$, and for all $a_{li}^g, b_{kj}^h\in\k$. Then for any $ g\in G_f, z\in T_f$, we have
$\sum\limits_{i=1}^m \sum\limits_{j=1}^n a_{ii}^g b_{jj}^{(z^{-1}gz)\triangleleft(z^{-1}\triangleright f)}=\sum\limits_{i=1}^m \sum\limits_{j=1}^n a_{ii}^{g}b_{jj}^{z^{-1}gz}.$
\end{theorem}
As a corollary, we can consider when $\Bbbk^G {}^\tau \#_{\sigma} \Bbbk F$ admits a quasitriangular structure in the case where $G$ and $F$ are both finite groups (see, Proposition \ref{prop:QT=OgOg}).

Moreover, we present characterizations of the coquasitriangular structure $R$ on $\Bbbk^G {}^\tau \#_{\sigma} \Bbbk F$ when it is a coquasitriangular Hopf algebra (see Proposition \ref{prop:CQTstrcuture}), and investigate its explicit values on specific elements.
For the specific case
$G=\mathbb{Z}_2$, we characterize the Grothendieck ring of the category of finite-dimensional comodules over
$\Bbbk^{\mathbb{Z}_2}{}^\tau\#_{\sigma}\Bbbk F$ and determine the conditions for its commutativity (see Proposition \ref{prop:Grz2} and Corollary \ref{coro:grz2commu} ). In particular, when
$F$ is an abelian group, we demonstrate that
the coquasitriangular structure on $\Bbbk^{\mathbb{Z}_2}{}^\tau\#_{\sigma}\Bbbk F$ must be one of two forms (see Proposition \ref{prop:FabelianR}).

The organization of this paper is as follows. In Section \ref{Section2}, we recall the definition of a Hopf algebra $\Bbbk^G {}^\tau \#_{\sigma} \Bbbk F$ and describe the category of its right comodules. We devote Section \ref{Section3} to provide some necessary conditions for $\Bbbk^G{}^\tau\#_{\sigma}\Bbbk F$ to be a coquasitriangular Hopf algebra. We further construct matched pairs of groups that fail the necessary conditions, thus providing new examples of cosemisimple Hopf algebras that are not coquasitriangular.
In Section \ref{Section 4}, we characterize the coquasitriangular structures on $\Bbbk^G{}^\tau\#_{\sigma}\Bbbk F$. Finally in Section \ref{Section5}, we apply our general results to the case
$G=\mathbb{Z}_2$, and obtain the characterization of the coquasitriangular structures on $\Bbbk^{\mathbb{Z}_2}{}^\tau\#_{\sigma}\Bbbk F$.

\section{The Hopf algebra $\Bbbk^G{}^\tau\#_{\sigma}\Bbbk F$}\label{Section2}
Throughout this paper, $\k$ denotes an algebraically closed field of characteristic zero, and all vector spaces are defined over $\k$. We denote the tensor product over $\k$ simply by $\otimes$. We shall employ Sweedler's sigma notation for coalgebras (\cite{Swe69}): the coproduct of any element $c$ is written as $\Delta(c) = \sum c_{(1)} \otimes c_{(2)}$. The reader is referred to \cite{Mon93}, \cite{Swe69} and \cite{EGNO15} for the basics about Hopf algebras and tensor categories.
\subsection{The definition of $\Bbbk^G{}^\tau\#_{\sigma}\Bbbk F$}
\begin{definition}\emph{(}\cite{AD95}\emph{)}
A short exact sequence of Hopf algebras is a sequence of Hopf algebras
and Hopf algebra maps
\begin{eqnarray}\label{ext}
\;\; K\xrightarrow{\iota} H \xrightarrow{\pi} A
\end{eqnarray}
such that
\begin{itemize}
  \item[(i)] $\iota$ is injective,
  \item[(ii)]  $\pi$ is surjective,
  \item[(iii)] $\ker(\pi)= HK^+$, where $K^+$ is the kernel of the counit of $K$,
  \item[(iv)] $K=\{x\in H \mid (\pi\otimes \id)\Delta (x)=1\otimes x\}$.
\end{itemize}
\end{definition}
An extension (\ref{ext}) above such that $K$ is commutative and $A$ is cocommutative is called abelian. The classification of abelian extensions $\k^G \xrightarrow{\iota} H \xrightarrow{\pi} \k F$, where $G$ and $F$ are finite groups and $\k^G$ is the dual Hopf algebra of $\k G$, was given by Masuoka \cite[Proposition 1.5]{Mas02}. Masuoka showed that any such $H$ is isomorphic to $\k^G \#_{\sigma, \tau} \k F$. Subsequently, the authors of \cite{YLZZ25} removed the finiteness restriction on $F$, thereby constructing a class of infinite-dimensional cosemisimple Hopf algebras $\Bbbk^G {}^\tau \#_{\sigma} \Bbbk F$ as abelian extensions of $\k F$ by $\k^G$.

In the following part, let $G$ be a finite group and $F$ an arbitrary group. Let $\{p_g \mid g \in G\}$ denote the canonical basis of $\k^G$ dual to the basis of group elements in $\k G$.
We now recall the Hopf algebra structure on $\Bbbk^G {}^\tau \#_{\sigma} \Bbbk F$, as established in \cite{YLZZ25}, which requires the following data:
\begin{itemize}
  \item [(1)]A matched pair of groups (\cite[Definition 2.1]{Tak81}), i.e., a quadruple $(F, G,\triangleleft, \triangleright )$, where $G\xleftarrow{\triangleleft} G\times F\xrightarrow{\triangleright} F$
are actions of groups on sets, satisfying the following conditions
 \begin{eqnarray*}
 g\triangleright(ff^\prime)=(g\triangleright f)((g\triangleleft f)\triangleright f^\prime),\;\;(gg^\prime)\triangleleft f=(g\triangleleft (g^\prime \triangleright f))(g^\prime\triangleleft f),
\end{eqnarray*}
for any $g, g^\prime\in G$, $f, f^\prime\in F$.
\item[(2)]A map $\sigma: \k G\times \k F\times \k F\rightarrow \k^{\times}$ is defined to satisfy the identities
\begin{eqnarray*}
\sigma(g; 1_F, f)=\sigma(g; f, 1_F)=\sigma(1_G; f, f^\prime)=1,
\end{eqnarray*}
and
\begin{eqnarray*}
\sigma(g\triangleleft f; f^\prime, f^{\prime\prime})\sigma(g; f, f^\prime f^{\prime\prime})=\sigma(g; f, f^\prime)\sigma(g; ff^\prime, f^{\prime\prime}),
\end{eqnarray*}
for any $g\in G$ and $f, f^\prime ,f^{\prime \prime}\in F.$
\item[(3)]A map $\tau :\k G\times \k G\times \k F\rightarrow \k^{\times}$ satisfying
\begin{eqnarray*}
\tau(1_G, g; f)=\tau(g, 1_G; f)=\tau(g, g^\prime; 1_F)=1,
\end{eqnarray*}
and
\begin{eqnarray*}
\tau(g, g^\prime; g^{\prime\prime}\triangleright f)\tau(gg^\prime, g^{\prime\prime}; f)=\tau(g, g^\prime g^{\prime\prime}; f)\tau(g^\prime, g^{\prime\prime}; f),
\end{eqnarray*}
for any $g, g^\prime, g^{\prime\prime}\in G$ and $f\in F.$
Moreover, $\sigma$ and $\tau$ satisfy the following compatible condition
  \begin{eqnarray*}
  &&\sigma(gg^\prime; f, f^\prime)\tau(g, g^\prime; ff^\prime)\\
  &=&\sigma(g; g^\prime \triangleright f, (g^\prime \triangleleft f)\triangleright f^\prime)\sigma(g^\prime; f, f^\prime)
  \tau(g, g^\prime; f)\tau(g\triangleleft (g^\prime \triangleright f), g^\prime\triangleleft f; f^\prime),
  \end{eqnarray*}
  for $g, g^\prime, g^{\prime\prime}\in G$ and $f, f^\prime, f^{\prime\prime}\in F$.
\end{itemize}
\begin{remark}\label{lem:actioninverse}\emph{(}\cite[Lemma 2.10]{YLZZ25}\emph{)}
Suppose $(F, G,\triangleleft, \triangleright )$ is a matched pair of groups. Then for any $f\in F, g\in G$, the following properties hold:
\begin{itemize}
  \item [(1)]$1_G \triangleright f =f$, $1_G \triangleleft f= 1_G$;
  \item [(2)]$g\triangleright 1_F=1_F$, $g\triangleleft 1_F=g$;
  \item [(3)]$(g\triangleright f)^{-1}=(g\triangleleft f)\triangleright f^{-1}$;
  \item [(4)]$(g\triangleleft f)^{-1}=g^{-1}\triangleleft (g\triangleright f)$.
\end{itemize}
\end{remark}
\begin{definition}\emph{(}\cite[Proposition 3.2]{YLZZ25}\emph{)}\label{def:H}
Based on the preceding notations, the Hopf algebra $\Bbbk^G {}^\tau \#_{\sigma} \Bbbk F$ coincides with $\k^G\otimes \k F$ as a vector space, and we denote the element $p_g\otimes f$ by $p_g\# f.$ For any $g, g^\prime\in G$ and $f, f^\prime \in F,$ the multiplication and comultiplication are defined as follows:
\begin{eqnarray*}
&&(p_g\# f)(p_{g^\prime}\# f^\prime)
=\delta_{g\triangleleft f, g^\prime}\sigma(g; f, f^\prime) p_{g}\#ff^\prime,\\
&&\Delta(p_g\# f)
=\sum_{x\in G}(\tau(gx^{-1}, x; f) p_{gx^{-1}} \#(x\triangleright f))\otimes ( p_x\# f).
\end{eqnarray*}
The unit is $ 1_{\k^G}\#  1_F$, the counit is given by \begin{eqnarray*}
\varepsilon(p_g\# f)=\delta_{g,1_G},
\end{eqnarray*}
and the antipode is
\begin{eqnarray*}
S(p_g\# f)=\sigma(g^{-1}; g\triangleright f, (g\triangleright f)^{-1})^{-1}\tau(g^{-1}, g; f)^{-1}p_{(g\triangleleft f)^{-1}}\# (g\triangleright f)^{-1}.
\end{eqnarray*}
\end{definition}

\begin{remark}
\begin{itemize}
\item[(1)]
Suppose $\sigma$ and $\tau$ are trivial, that is, $\tau(g, g^\prime; f)=\sigma(g; f, f^\prime)=1$ for any $g, g^\prime\in G, f, f^\prime\in F$.  In this case, we simply write $\Bbbk^G  \# \Bbbk F$ for $\Bbbk^G {}^\tau \#_{\sigma} \Bbbk F$.
\item[(2)] The abelian extension $\k^G \xrightarrow{\iota} \Bbbk^G {}^\tau \#_{\sigma} \Bbbk F \xrightarrow{\pi} \k F$ is said to be \textit{central}, if $\k^G$ is contained in the center of $\Bbbk^G {}^\tau \#_{\sigma} \Bbbk F$. An equivalent condition is that $g\triangleleft f=g$ for any $g\in G, f\in F.$
\end{itemize}
\end{remark}
\begin{lemma}\emph{(}\cite[Proposition 3.5]{YLZZ25}\emph{)}\label{lem:cose}
The Hopf algebra $\k^G{}^\tau\#_{\sigma}\k F$ given in Definition \ref{def:H} is cosemisimple.
\end{lemma}

\subsection{The comodules over $\Bbbk^G{}^\tau\#_{\sigma}\Bbbk F$}
In this subsection, we present some results from \cite{YLZZ25} concerning the simple right comodules over  $\k^G{}^\tau\#_{\sigma}\k F$.

Drawing on the work of \cite{JM09}, we adopt the following definitions.
\begin{definition}
Suppose $(F, G,\triangleleft, \triangleright )$ is a matched pair of groups. For any $f, f^\prime\in F$, we define the followings:
\begin{itemize}
  \item [(1)]The orbits $O_{f}=\{g\triangleright f\mid g\in G\}$;
  \item[(2)]The sets $O_fO_{f^\prime}=\{xy\mid x\in O_f, y\in O_{f^\prime}\};$
  \item [(3)]The stabilizers $G_{f}=\{g\in G\mid g\triangleright f=f\}$;
  \item [(4)]A complete set $T_f$ of right coset representatives of $G_f$ in $G$. For convenience, we assume $1_G\in T_f.$
\end{itemize}
\end{definition}
Recall that the \textit{cotensor product} (see \cite{Tak77}) of a right comodule $M$ and a left comodule $N$ over a coalgebra $C$ is defined as
\begin{eqnarray*}
M \Box_C N = \ker( M\otimes N \xrightarrow{\rho_M\otimes \id-\id\otimes \rho_N} M\otimes C\otimes N),
\end{eqnarray*}
where $\rho_M$ and $\rho_N$ are the respective comodule structure maps.

For any $f\in F$, let $\k^{G_f}{}^\tau\#\k F$ denote the linear space spanned by $\{p_g\# f^\prime\mid g\in G_f, f^\prime\in F\}$. This space can be endowed with a coalgebra structure via the crossed coproduct, whose comultiplication and counit are given by
$$
\Delta(p_g\# f^\prime)=\sum_{x\in G_f}(\tau(gx^{-1}, x; f^\prime) p_{gx^{-1}} \#(x\triangleright f^\prime))\otimes ( p_x\# f^\prime),
$$
and
$$
\varepsilon(p_g\# f^\prime)=\delta_{g,1_G},
$$
where $g\in G_f, f^\prime\in F$.

It was shown in \cite{YLZZ25} that all simple right $\k^G{}^\tau\#_{\sigma}\k F$-comodules can be described as induced comodules from $\k_{\tau_f}^{G_f}$, where $\k_{\tau_f}^{G_f}$ is the coalgebra over $\k^{G_f}$ with the revised comultiplication
$$\Delta(p_g)=\sum_{x\in G_f}\tau(gx^{-1}, x; f) p_{gx^{-1}}\otimes p_{x}.$$
We now state the following lemma.
\begin{lemma}\emph{(}\cite[Lemma 4.2 and Theorem 4.3]{YLZZ25}\emph{)}\label{lem:simplecomod}
Let $H=\k^G{}^\tau\#_{\sigma}\k F$, and fix an element $f\in F.$ Let $H^\prime=\k^{G_f}{}^\tau\#\k F$ and $\k_{\tau_f}^{G_f}$ be the coalgebras defined above.
\begin{itemize}
  \item [(1)]Let $(V, \rho)$ be a right $\k_{\tau_f}^{G_f}$-comodule, and let $V^\prime=V\otimes \k f$. Then $V^\prime$ becomes a right $H^\prime$-comodule by defining, for each $v\in V$,
      $$
      \rho^\prime(v\otimes f)=\sum_{g\in G_{f}} v_g\otimes f\otimes p_g\# f,
      $$
      where $\rho(v)=\sum_{g\in G_{f}} v_g\otimes p_g.$
  \item [(2)]Let $\tilde{V}:=V^\prime \Box_{H^\prime} H=(V\otimes \k f)\Box_{H^\prime} H$. Then we may write $\tilde{V}=\sum_{z\in T_f} V\otimes \k f\otimes z.$ Moreover, $\tilde{V}$ becomes a right $H$-comodule whose comodule structure is induced by the comultiplication of $H$.
      For any $x\in G$, suppose $x=g_xz_x,$ where $g_x\in G_f$ and $z_x\in T_f$. Then the comodule structure on $\tilde{V}$ is given by the following formula:
      \begin{eqnarray*}
      &&\tilde{\rho}(v\otimes f\otimes z)\\
      &=&\sum_{x\in G} \tau(z_x^{-1},g_x^{-1};f)^{-1}\tau(z_x^{-1}g_x^{-1}z, z^{-1};f)v_{g_x^{-1}}\otimes f\otimes z_x \otimes p_{x^{-1}z}\#(z^{-1}\triangleright f),
      \end{eqnarray*}
      for any $v\in V$ and $z\in T_f$.
    \item [(3)]If $V$ is a simple right $\k_{\tau_f}^{G_f}$-comodule, then $\tilde{V}$ is a simple right $H$-comodule.
  \item [(4)]Every simple right $H$-comodule is isomorphic to $\tilde{V}$ for some simple right $\k_{\tau_f}^{G_f}$-comodule, where $f$ ranges over a choice of one element in each $G$-orbit of $F$.
\end{itemize}
\end{lemma}
Let $C$ be a coalgebra over $\k$ and $(M, \rho)$ be a right $C$-comodule. Pick a basis $\{m_1, m_2, \cdots, m_n\}$ for $M$ and let $\{m_1^*, m_2^*, \cdots, m_n^*\}$ be the basis of $W^*$ dual to this basis. The \textit{character} $\chi(M)$ (\cite[Section 1]{Lar71}) of the comodule $M$ is defined by
$$
\chi(M)=\sum_{i=1}^n (m_i^*\otimes \id)\rho(m_i).
$$
If $M$ is a simple right $C$-comodule, then $\chi(M)$ is called an \textit{irreducible character}.

\begin{lemma}\emph{(}\cite[Proposition 5.2]{YLZZ25}\emph{)}\label{lem:character}
Let $H=\k^G{}^\tau\#_{\sigma}\k F$, $V=\span\{v^{(1)}, \cdots, v^{(m)}\}$ be a simple right $\k_{\tau_f}^{G_f}$-comodule and $\tilde{V}=(V\otimes \k f)\Box_{\k^{G_f}{}^\tau\#\k F} H$ be the induced right $H$-comodule. Suppose for any $1\leq i\leq m$, we have $$\rho(v^{(i)})=\sum_{j=1}^m v^{(j)}\otimes (\sum_{g\in G_f}a_{ji}^g p_g),$$
where $a_{ij}^g\in \k$ for any $1\leq i, j\leq m$ and $g\in G_f.$ Then the irreducible character $\chi(\tilde{V})$ of $\tilde{V}$ is
$$\sum\limits_{i=1}^m\sum\limits_{z\in T_f}\sum\limits_{g\in G_f}\tau(z^{-1},g;f)^{-1}\tau(z^{-1}gz, z^{-1};f)a_{ii}^{g}p_{z^{-1}gz}\#(z^{-1}\triangleright f).$$
\end{lemma}
\section{Necessary Conditions for $\Bbbk^G{}^\tau\#_{\sigma}\Bbbk F$ to be coquasitriangular}\label{Section3}
In this section, let $G$ be a finite group and $F$ an arbitrary group.
The aim of this section is to provide some necessary conditions for $\Bbbk^G{}^\tau\#_{\sigma}\Bbbk F$ to be a coquasitriangular Hopf algebra. Moreover, we construct matched pairs of groups that fail to satisfy the necessary conditions, thereby obtaining new examples of cosemisimple Hopf algebras which are not coquasitriangular.

Recall that a \textit{coquasitriangular} Hopf algebra (\cite{Mon93}) is a pair $(H, R)$, where $H$ is a Hopf algebra over $\Bbbk$ and $R: H\otimes H\rightarrow \Bbbk$ is a convolution-invertible bilinear form on $H$ satisfying the following axioms:
\begin{eqnarray}\label{CQT1}
R(g,hl)=\sum R(g_{(1)}, l)R(g_{(2)}, h), \;\; R(gh,l)=\sum R(g, l_{(1)})R(h, l_{(2)})
\end{eqnarray}
and
\begin{eqnarray}\label{CQT3}
\sum R(g_{(1)}, h_{(1)})g_{(2)}h_{(2)}=\sum h_{(1)}g_{(1)} R(g_{(2)}, h_{(2)}),
\end{eqnarray}
for any $g, h, l\in H.$ If
\begin{eqnarray}\label{CQT4}
\sum R(g_{(1)}, h_{(1)}) R(h_{(2)}, g_{(2)})=\varepsilon (g) \varepsilon (h),
\end{eqnarray}
then $(H, R)$ is called \textit{cotriangular}.

To proceed, we recall the following lemma.
\begin{lemma}\emph{(}\cite[Theorem 10.4.2]{Mon93}\emph{)}\label{lem:braided}
Let $H$ be a Hopf algebra over $\k$. Then $(H, R)$ is coquasitriangular if and only if the category of right $H$-comodules is a braided monoidal category.
\end{lemma}
Let $\mathbb{Z}_+$ be the set of nonnegative integers, $A$ be an associative ring with unit which is free as a $\mathbb{Z}$-module.
Recall that a $\mathbb{Z}_+$\textit{-basis} of $A$ is a basis $B=\{b_{i}\}_{i\in I}$ such that $b_ib_j=\sum_{k\in I}c_{ij}^kb_k$, where $c_{ij}^k\in\mathbb{Z}_+$. A $\mathbb{Z}_+$\textit{-ring} $A$ is a ring with a fixed $\mathbb{Z}_+$-basis, and $A$ is a \textit{unital} $\mathbb{Z}_+$-ring if in addition $1$ is a basis element.

Suppose $H$ is a cosemisimple Hopf algebra over $\k$. Let $\mathcal{F}$ be the free abelian group generated by isomorphism classes of finite-dimensional right $H$-comodules and $\mathcal{F}_0$ the subgroup of $\mathcal{F}$ generated by all expressions $[Y]-[X]-[Z]$, where $0\rightarrow X\rightarrow Y\rightarrow Z\rightarrow0$ is a short exact sequence of finite-dimensional right $H$-comodules.
Recall that the \textit{Grothendieck group} $\operatorname{Gr}(H$-comod$)$ of the category of finite-dimensional right $H$-comodules is defined by $$\operatorname{Gr}(H\text{-comod}):=\mathcal{F}/\mathcal{F}_0.$$

Then \cite[Proposition 4.5.4]{EGNO15} implies that the Grothendieck ring $\operatorname{Gr}(H\text{-comod})$ is a unital $\Bbb{Z}_+$-ring. Its $\mathbb{Z}_+$-basis $\mathcal{V}$ is given by the isomorphism classes of simple right $H$-comodules.

Having established this, let $\Lambda$ be the set of irreducible characters corresponding to the simple comodules in $\mathcal{V}$. We define a ring structure on the $\mathbb{Z}$-module $\mathbb{Z}\Lambda$ by specifying the product on characters: for any $\chi(M), \chi(N) \in \Lambda$,
$$\chi(M)\chi(N):=\chi(M\otimes N).$$
According to \cite[Subsection 2.2]{YLZZ25}, $\mathbb{Z}\Lambda$ is a unital $\mathbb{Z}_+$ ring with $\mathbb{Z}_+$-basis $\Lambda$.

Moreover, we have the following result:
\begin{lemma}\emph{(}\cite[Lemma 2.8]{YLZZ25}\emph{)}\label{lem:Gr=character}
Let $H$ be a cosemisimple Hopf algebra over $\k$ and $\Lambda$ be the set of all the irreducible characters of simple right $H$-comodules. Then \begin{eqnarray*}
F:\operatorname{Gr}(H\text{-comod})&\rightarrow&\mathbb{Z}\Lambda,\\
 M\;\;\;\;&\mapsto &\chi(M).
\end{eqnarray*} is an isomorphism between unital $\mathbb{Z}_+$-rings.
\end{lemma}
From the above lemmas, the following theorem is derived, describing the constraints imposed by the coquasitriangular structure on matched pairs of groups.
\begin{theorem}\label{prop:OfOf}
Suppose $\Bbbk^G{}^\tau\#_{\sigma}\Bbbk F$ admits a coquasitriangular structure. Then
$O_fO_{f^\prime}=O_{f^\prime}O_f$ for all $f, f^\prime\in F.$
\end{theorem}
\begin{proof}
Suppose $V = \operatorname{span}\{v^{(1)}, \dots, v^{(m)}\}$ and $W = \operatorname{span}\{w^{(1)}, \dots, w^{(n)}\}$ are simple right comodules over $\Bbbk_{\tau_f}^{G_f}$ and $\Bbbk_{\tau_{f^\prime}}^{G_{f^\prime}}$, with coactions $$\rho(v^{(i)})=\sum_{l=1}^m v^{(l)}\otimes (\sum_{g\in G_f}a_{li}^g p_g),$$ and $$\rho(w^{(j)})=\sum_{k=1}^n w^{(k)}\otimes (\sum_{h\in G_{f^\prime}}b_{kj}^{h} p_{h}),$$
defined for all $1\leq l, i\leq  m, 1\leq k, j\leq n$, and for all $a_{li}^g, b_{kj}^h\in\k$.
Let $\tilde{V}$ and $\tilde{W}$ be the induced right $\k^G{}^\tau\#_{\sigma}\k F$-comodules of $V$ and $W$, respectively (see Lemma \ref{lem:simplecomod}).
An immediate consequence of Lemma \ref{lem:character} is that
\begin{eqnarray*}
&&\chi(\tilde{V})\chi(\tilde{W})\\
&=&(\sum\limits_{i=1}^m\sum\limits_{z\in T_f}\sum\limits_{g\in G_f}\tau(z^{-1},g;f)^{-1}\tau(z^{-1}gz, z^{-1};f)a_{ii}^{g}p_{z^{-1}gz}\#(z^{-1}\triangleright f))\\
&&(\sum\limits_{j=1}^n\sum\limits_{y\in T_{f^\prime}}\sum\limits_{h\in G_{f^\prime}}\tau(y^{-1},h;f^\prime)^{-1}\tau(y^{-1}hy, y^{-1};f^\prime)b_{jj}^{h}p_{y^{-1}hy}\#(y^{-1}\triangleright f^{\prime}))\\
&=&\sum\limits_{i=1}^m\sum\limits_{z\in T_f}\sum\limits_{g\in G_f}\sum\limits_{j=1}^n\sum\limits_{y\in T_{f^\prime}}\sum\limits_{h\in G_{f^\prime}}
\tau(z^{-1},g;f)^{-1}\tau(z^{-1}gz, z^{-1};f)\tau(y^{-1},h;f^\prime)^{-1}\\
&&\tau(y^{-1}hy, y^{-1};f^\prime)a_{ii}^{g}b_{jj}^{h}\delta_{(z^{-1}gz)\triangleleft (z^{-1}\triangleright f), y^{-1}hy} \sigma(z^{-1}gz; z^{-1}\triangleright f, y^{-1}\triangleright f^{\prime})\\
&&p_{z^{-1}gz}\# (z^{-1}\triangleright f)(y^{-1}\triangleright f^{\prime}),
\end{eqnarray*}
and
\begin{eqnarray*}
&&\chi(\tilde{W})\chi(\tilde{V})\\
&=&(\sum\limits_{j=1}^n\sum\limits_{y\in T_{f^\prime}}\sum\limits_{h\in G_{f^\prime}}\tau(y^{-1},h;f^\prime)^{-1}\tau(y^{-1}hy, y^{-1};f^\prime)b_{jj}^{h}p_{y^{-1}hy}\#(y^{-1}\triangleright f^{\prime}))\\
&&(\sum\limits_{i=1}^m\sum\limits_{z\in T_f}\sum\limits_{g\in G_f}\tau(z^{-1},g;f)^{-1}\tau(z^{-1}gz, z^{-1};f)a_{ii}^{g}p_{z^{-1}gz}\#(z^{-1}\triangleright f))\\
&=&\sum\limits_{i=1}^m\sum\limits_{z\in T_f}\sum\limits_{g\in G_f}\sum\limits_{j=1}^n\sum\limits_{y\in T_{f^\prime}}\sum\limits_{h\in G_{f^\prime}}
\tau(z^{-1},g;f)^{-1}\tau(z^{-1}gz, z^{-1};f)\tau(y^{-1},h;f^\prime)^{-1}\\
&&\tau(y^{-1}hy, y^{-1};f^\prime)a_{ii}^{g}b_{jj}^{h}\delta_{(y^{-1}hy)\triangleleft (y^{-1}\triangleright f^\prime), z^{-1}gz} \sigma(y^{-1}hy; y^{-1}\triangleright f^\prime, z^{-1}\triangleright f)\\
&&p_{y^{-1}hy}\# (y^{-1}\triangleright f^{\prime}) (z^{-1}\triangleright f).
\end{eqnarray*}
By combining Lemmas \ref{lem:braided} and \ref{lem:Gr=character}, we conclude that $$\chi(\tilde{V})\chi(\tilde{W})=\chi(\tilde{W})\chi(\tilde{V}).$$
Since the set $\{p_x\# u \mid x \in G, u \in F\}$ is linearly independent, the equality above implies that the coefficients of each basis vector agree. In particular, we compare the coefficients of $\sum\limits_{u\in F}p_1\# u$ in the respective operator products $\chi(\tilde{V})\chi(\tilde{W})$ and $\chi(\tilde{W})\chi(\tilde{V})$. Note that $a_{ii}^{1_G} = b_{jj}^{1_G} = 1$ and that $z^{-1}gz = 1_G$ forces $g = 1_G$. The coefficient in $\chi(\tilde{V})\chi(\tilde{W})$ is
\begin{eqnarray*}
&&\sum\limits_{i=1}^m\sum\limits_{z\in T_f}\sum\limits_{j=1}^n\sum\limits_{y\in T_{f^\prime}}\sum\limits_{h\in G_{f^\prime}}
\tau(y^{-1},h;f^\prime)^{-1}
\tau(y^{-1}hy, y^{-1};f^\prime)a_{ii}^{1_G}b_{jj}^{h}\delta_{1_G, y^{-1}hy} \\
&&p_{1_G}\# (z^{-1}\triangleright f)(y^{-1}\triangleright f^{\prime})\\
&=& mn(\sum\limits_{z\in T_f}\sum\limits_{y\in T_{f^\prime}}  p_{1_G}\# (z^{-1}\triangleright f)(y^{-1} \triangleright f^{\prime})),
\end{eqnarray*}
while in $\chi(\tilde{W})\chi(\tilde{V})$, it is
\begin{eqnarray*}
mn(\sum\limits_{z\in T_f}\sum\limits_{y\in T_{f^\prime}}
p_{1_G}\# (y^{-1}\triangleright f^{\prime})(z^{-1} \triangleright f)).
\end{eqnarray*}
It therefore follows that
$$
\sum\limits_{z\in T_f}\sum\limits_{y\in T_{f^\prime}}   (z^{-1}\triangleright f)(y^{-1} \triangleright f^{\prime})=
\sum\limits_{z\in T_f}\sum\limits_{y\in T_{f^\prime}}(y^{-1}\triangleright f^{\prime})(z^{-1} \triangleright f).
$$
Now, for an arbitrary $x\in G$, we consider its decomposition $x=g_xz_x,$ where $g_x\in G_f$ and $z_x\in T_f$.
Then, the action of $x^{-1}$ on
$f$ is given by $$x^{-1}\triangleright f=(z_x^{-1} g_x^{-1})  \triangleright f=z_x^{-1}\triangleright f,$$
where the last equality holds because $g_x^{-1}\in G_f$ stabilizes $f$.
This identity implies that
\begin{eqnarray}\label{eq:Of}
O_f= \{z^{-1}\triangleright f\mid z\in  T_f\},
\end{eqnarray}from which we conclude that
$$O_fO_{f^\prime}=O_{f^\prime }O_{f}. $$
\end{proof}
A direct consequence of the proposition above is:
\begin{corollary}\label{coro:Fnonabelian}
Let $F$ be a non-abelian group. Suppose that $g\triangleright f=f$ for any $g\in G, f\in F,$ then $\Bbbk^G{}^\tau\#_{\sigma}\Bbbk F$ admits no coquasitriangular structure.
\end{corollary}
\begin{proof}
In this case, $O_f=\{f\}$ for any $f\in F.$ Since $F$ is non-abelian, it follows that $O_fO_{f^\prime} \neq O_{f^\prime} O_f$. According to Theorem \ref{prop:OfOf}, the Hopf algebra $\Bbbk^G{}^\tau\#_{\sigma}\Bbbk F$ admits no coquasitriangular structure.
\end{proof}

\begin{example}\label{ex:1}
Let $\mathbb{Z}_{n}=\langle g\mid g^{n}=1\rangle$ for some $n\geq 2$ and $\Bbb{D}_\infty$ denote the infinite dihedral group $\langle x,y\mid x^2=1, xyx=y^{-1}\rangle$. Define group actions $\mathbb{Z}_n\xleftarrow{\triangleleft}\mathbb{Z}_n\times \Bbb{D}_\infty \xrightarrow{\triangleright}\Bbb{D}_\infty$ on the sets by
$$
g^{i}\triangleright x^{j}y^{k} x^{l}=x^{j}y^{k} x^{l},\;\;
g^i\triangleleft x^{j}y^{k} x^{l}=
\left\{
\begin{aligned}
g^{i}~,~~~ &\text{if} ~~~  k ~~\text{is even}; \\
g^{n-i},~~~&\text{if} ~~~  k ~~\text{is odd},
\end{aligned}
\right.
$$
where $0\leq i\leq n-1, 0\leq j,l\leq 1, k\in \mathbb{Z}.$
 Clearly, $(\Bbb{D}_\infty, \mathbb{Z}_n)$ together with group actions $\mathbb{Z}_n\xleftarrow{\triangleleft}\mathbb{Z}_n\times \Bbb{D}_\infty \xrightarrow{\triangleright}\Bbb{D}_\infty$ on the sets is a matched pair. In such a case, by Corollary \ref{coro:Fnonabelian}, $\Bbbk^{\mathbb{Z}_n}\#\Bbbk \Bbb{D}_\infty$ admits no coquasitriangular structure.
\end{example}
\begin{example}\label{ex:q8dinf}
Let $\mathcal{Q}_8$ be the quaternion group $\langle r, s\mid r^4=1, r^2=s^2, sr=r^{-1}s\rangle$ and $\Bbb{D}_\infty$ denote the infinite dihedral group $\langle x,y\mid x^2=1, xyx=y^{-1}\rangle$. Define group actions $\mathcal{Q}_8\xleftarrow{\triangleleft}\mathcal{Q}_8\times \Bbb{D}_\infty \xrightarrow{\triangleright}\Bbb{D}_\infty$ on the sets by
$$
u\triangleright v=v,\;\;r\triangleleft x=s,\;\; s\triangleleft x=r,\;\; r\triangleleft y=s,\;\; s\triangleleft y=r,
$$
where $u\in \mathcal{Q}_8, v\in \Bbb{D}_\infty.$
Then, the pair $(\Bbb{D}_\infty, \mathcal{Q}_8)$ forms a matched pair of groups.
It follows from Corollary \ref{coro:Fnonabelian} that $\k^{\mathcal{Q}_8}\# \Bbb{D}_\infty$ is not a coquasitriangular Hopf algebra.
\end{example}
Under certain conditions, we can derive further constraints imposed by the coquasitriangular structure on the mutual actions between the two groups within a matched pair.
\begin{proposition}\label{prop:action}
Suppose $\Bbbk^G{}^\tau\#_{\sigma}\Bbbk F$ is a coquasitriangular Hopf algebra, where $G$ is an abelian group, $\tau$ is trivial. Let $f, f^\prime \in F$.
\begin{itemize}
\item[(1)]If $g\in G_f, g\notin G_{f^\prime}$, then $g\triangleleft f^{\prime\prime} \notin G_{f^\prime}$ for any $f^{\prime\prime}\in O_f;$
\item[(2)]Suppose $g\in G_f\cap G_{f^\prime}$.
Then there exists some $f^{\prime\prime}\in O_f$ such that $g\triangleleft f^{\prime\prime}\in G_{f^\prime}$ if and only if there exists some $f^{\prime\prime\prime} \in O_{f^\prime}$ such that $g\triangleleft f^{\prime\prime\prime}\in G_{f}$.
\end{itemize}
\end{proposition}
\begin{proof}
\begin{itemize}
\item[(1)]Since $G$ is abelian, the dual $\k^G$ also forms an abelian group algebra. As a consequence, every simple right $\k^{G_f}$-comodule is $1$-dimensional.
Suppose $V = \operatorname{span}\{v\}$ and $W = \operatorname{span}\{w\}$ are simple right comodules over $\Bbbk^{G_f}$ and $\Bbbk^{G_{f^\prime}}$, with coactions $$\rho(v)= v\otimes (\sum_{g\in G_f}a^g p_g),$$ and $$\rho(w)= w\otimes (\sum_{h\in G_{f^\prime}}b^{h} p_{h}),$$
defined for all $a^g, b^h\in\k$.
Let $\tilde{V}$ and $\tilde{W}$ be the induced right $\k^G{}^\tau\#_{\sigma}\k F$-comodules of $V$ and $W$, respectively.
Analogously to the proof of Theorem \ref{prop:OfOf}, we deduce that
\begin{eqnarray*}
&&\chi(\tilde{V})\chi(\tilde{W})\\
&=&\sum\limits_{z\in T_f}\sum\limits_{g\in G_f}\sum\limits_{y\in T_{f^\prime}}\sum\limits_{h\in G_{f^\prime}}
a^{g}b^{h}\delta_{g\triangleleft (z^{-1}\triangleright f),h} \sigma(g; z^{-1}\triangleright f, y^{-1}\triangleright f^{\prime})\\
&&p_{g}\# (z^{-1}\triangleright f)(y^{-1}\triangleright f^{\prime})\\
&=&\sum\limits_{z\in T_f}\sum\limits_{g\in G_f}\sum\limits_{y\in T_{f^\prime}}\sum\limits_{h\in G_{f^\prime}}
a^{g}b^{h}\delta_{h\triangleleft (y^{-1}\triangleright f^\prime), g} \sigma(h; y^{-1}\triangleright  f^\prime, z^{-1}\triangleright f)\\
&&p_{h}\# (y^{-1}\triangleright f^{\prime}) (z^{-1}\triangleright f)\\
&=&
\chi(\tilde{W})\chi(\tilde{V}).
\end{eqnarray*}
In view of Theorem \ref{prop:OfOf}, for any $z\in T_f, y\in T_{f^\prime}$, we assume that  $$(z^{-1}\triangleright f)(y^{-1}\triangleright f^{{\prime}})=(u_y^{-1}\triangleright f^\prime)(u_z^{-1}\triangleright f)$$ holds for some $u_z\in T_{f}, u_y\in T_{f^\prime}$.
 If $g\in G_f, g\notin G_{f^\prime}$, by comparing the coefficients of $p_g\# (z^{-1}\triangleright f)(y^{-1}\triangleright f^{{\prime}})$, we have
\begin{eqnarray*}
&&\sum\limits_{h\in G_{f^\prime}}a^{g}b^{h}
\delta_{g\triangleleft (z^{-1}\triangleright f),h} \sigma(g; z^{-1}\triangleright f, y^{-1}\triangleright f^{\prime})=0.
\end{eqnarray*}
It follows that
\begin{eqnarray*}
\sum\limits_{h\in G_{f^\prime}}a^{g}b^{h}
\delta_{g\triangleleft (z^{-1}\triangleright f),h}=0.
\end{eqnarray*}
Suppose that there exists some $h\in G_{f^\prime}$ such that $$h=g\triangleleft (z^{-1}\triangleright f).$$
We have $$a^gb^{g\triangleleft (z^{-1}\triangleright f) }=0.$$
Note that $\k^{G_f}$ can be decomposed into a direct sum of one-dimensional subcoalgebras.
Thus for any $x\in G_f$, there exists a simple right $\k^{G_f}$-comodule $V_x = \operatorname{span}\{v_x\}$ whose coaction is given by
$$\rho(v_x)=v_x\otimes (\sum_{g\in G_f} a_x^g p_g),$$
where $a_x^{g}\in \k$ for all $g\in G_f,$ and $a_x^x\neq 0.$
A similar statement holds for $\k^{G_{f^\prime}}$-comodule $W_x = \operatorname{span}\{w_x\}$. Setting
$
V=V_g, W=W_{g\triangleleft (z^{-1}\triangleright f)}
$
leads to a contradiction. Thus, by (\ref{eq:Of}), if $g\in G_f, g\notin G_{f^\prime}$, then $g\triangleleft f^{\prime\prime} \notin G_{f}$ for any $f^{\prime\prime}\in O_f.$
\item[(2)]
Similarly, we compare the coefficients of
 $p_g\# (z^{-1}\triangleright f)(y^{-1}\triangleright f^{{\prime}})$ for any $z\in T_f, y\in T_{f^\prime}$. If there exists some $f^{\prime\prime}\in O_f$ such that $g\triangleleft f^{\prime\prime}\in G_{f^\prime}$, then (\ref{eq:Of}) yeilds that there exists some $z\in T_f$ such that $$g\triangleleft(z^{-1}\triangleright f)=h\in G_{f^\prime}.$$
 We have
 \begin{eqnarray*}
a_g^{g}b_{g\triangleleft(z^{-1}\triangleright f)}^{g\triangleleft(z^{-1}\triangleright f)} \sigma(g; z^{-1}\triangleright f, y^{-1}\triangleright f^{\prime})\neq0
\end{eqnarray*}
This implies that the coefficient of
 $p_g\# (z^{-1}\triangleright f)(y^{-1}\triangleright f^{{\prime}})$ in $\chi(\tilde{V_g})\chi(\tilde{W}_{g\triangleleft(z^{-1}\triangleright f)})$ is nonzero. By the fact that $$\chi(\tilde{V_g})\chi(\tilde{W}_{g\triangleleft(z^{-1}\triangleright f)})=\chi(\tilde{W}_{g\triangleleft(z^{-1}\triangleright f)})\chi(\tilde{V_g}),$$ we know that there exists some $y\in T_{f^\prime}$ such that $g\triangleleft (y^{-1}\triangleright f^\prime))\in G_{f}$. The opposite direction can be proved similarly. Thus there exists some $f^{\prime\prime}\in O_f$ such that $g\triangleleft f^{\prime\prime}\in G_{f^\prime}$ if and only if there exists some $f^{\prime\prime\prime} \in O_{f^\prime}$ such that $g\triangleleft f^{\prime\prime\prime}\in G_{f}$.
\end{itemize}
\end{proof}

\begin{corollary}
Suppose $\Bbbk^G{}^\tau\#_{\sigma}\Bbbk F$ is a coquasitriangular Hopf algebra, where $G, F$ are abelian groups, $\tau$ is trivial.
If the extension $\k^G\rightarrow \k^G{}^\tau\#_\sigma \k F\rightarrow \k F $ is central, then we have $\sigma(g; f, f^\prime)=\sigma(g; f^\prime, f)$ for any $f, f^\prime \in F, g\in G_f\cap G_{f^\prime}$.
\end{corollary}
\begin{proof}
By an argument similar to the proof of Proposition \ref{prop:action}, one can show that
\begin{eqnarray*}
&&a_g^{g}b_g^{g}\sigma(g; z^{-1}\triangleright f, y^{-1}\triangleright f^{\prime})p_{g}\# (z^{-1}\triangleright f)(y^{-1}\triangleright f^{\prime})\\
&=&
a_g^{g}b_g^{g} \sigma(g; y^{-1}\triangleright f^\prime, z^{-1}\triangleright f)p_{g}\# (y^{-1}\triangleright f^{\prime}) (z^{-1}\triangleright f),
\end{eqnarray*}
where $g\in G_f\cap G_{f^\prime}$, and $a^g_g, b^g_g$ are as defined in the proof of Proposition \ref{prop:action}.
It follows that $$\sigma(g; f, f^\prime)=\sigma(g; f^\prime, f).$$
\end{proof}

In fact, the coquasitriangular structure of
$\Bbbk^G{}^\tau\#_\sigma\Bbbk F$ is also related to the comodule structure of the comodules over both $\k_{\tau_f}^{G_f}$ and $\k^G$ for any $f\in F.$.
\begin{theorem}\label{prop:cqtandcomod}
Let $\Bbbk^G{}^\tau\#_\sigma\Bbbk F$ be a coquasitriangular Hopf algebra. For any $f\in F$, if $V = \operatorname{span}\{v^{(1)}, \dots, v^{(m)}\}$ and $W = \operatorname{span}\{w^{(1)}, \dots, w^{(n)}\}$ are simple right comodules over $\Bbbk^{G_f}_{\tau_f}$ and $\k^G$, with coactions $$\rho(v^{(i)})=\sum_{l=1}^m v^{(l)}\otimes (\sum_{g\in G_f}a_{li}^g p_g),$$ and $$\rho(w^{(j)})=\sum_{k=1}^n w^{(k)}\otimes (\sum_{h\in G}b_{kj}^{h} p_{h}),$$
defined for all $1\leq l, i\leq  m, 1\leq k, j\leq n$, and for all $a_{li}^g, b_{kj}^h\in\k$. Then for any $ g\in G_f, z\in T_f$, we have
$\sum\limits_{i=1}^m \sum\limits_{j=1}^n a_{ii}^g b_{jj}^{(z^{-1}gz)\triangleleft(z^{-1}\triangleright f)}=\sum\limits_{i=1}^m \sum\limits_{j=1}^n a_{ii}^{g}b_{jj}^{z^{-1}gz}.$
\end{theorem}
\begin{proof}
Setting $f^\prime =1$ in Theorem \ref{prop:OfOf} and
following the same steps as in its proof, we obtain
\begin{eqnarray*}
&&\chi(\tilde{V})\chi(\tilde{W})\\
&=&\sum\limits_{i=1}^m\sum\limits_{z\in T_f}\sum\limits_{g\in G_f}\sum\limits_{j=1}^n\sum\limits_{h\in G}
\tau(z^{-1},g;f)^{-1}\tau(z^{-1}gz, z^{-1};f)\\
&&a_{ii}^{g}b_{jj}^{h}\delta_{(z^{-1}gz)\triangleleft (z^{-1}\triangleright f), h}
p_{z^{-1}gz}\# (z^{-1}\triangleright f)\\
&=&\sum\limits_{i=1}^m\sum\limits_{z\in T_f}\sum\limits_{g\in G_f}\sum\limits_{j=1}^n\sum\limits_{h\in G}
\tau(z^{-1},g;f)^{-1}\tau(z^{-1}gz, z^{-1};f)a_{ii}^{g}b_{jj}^{h}\delta_{h, z^{-1}gz}
p_{h}\#  (z^{-1}\triangleright f)\\
&=&\chi(\tilde{W})\chi(\tilde{V}).
\end{eqnarray*}
It follows that
$$\sum\limits_{i=1}^m\sum\limits_{j=1}^n a_{ii}^{g}b_{jj}^{(z^{-1}gz)\triangleleft (z^{-1}\triangleright f)}= \sum\limits_{i=1}^m \sum\limits_{j=1}^n a_{ii}^{g}b_{jj}^{z^{-1}gz}.$$
\end{proof}

\begin{corollary}
Let $\Bbbk^G{}^\tau\#_\sigma\Bbbk F$ be a coquasitriangular Hopf algebra, where $\tau$ is trivial. Suppose $W = \operatorname{span}\{w^{(1)}, \dots, w^{(n)}\}$ is a simple right comodule over $\k^G$, with coaction $$\rho(w^{(j)})=\sum_{k=1}^n w^{(k)}\otimes (\sum_{h\in G}b_{kj}^{h} p_{h}),$$
defined for all $ b_{kj}^h\in\k$, where $ 1\leq k, j\leq n, h\in G$. Then $ \sum\limits_{j=1}^n  b_{jj}^{(z^{-1}gz)\triangleleft(z^{-1}\triangleright f)}= \sum\limits_{j=1}^n b_{jj}^{z^{-1}gz}$ hold for any $ g\in G_f, z\in T_f$.
\end{corollary}
\begin{proof}
Since $\tau$ is trivial, it follows that $\k^{G_f}_{\tau_f}=\k^{G_f}$ is a Hopf algebra.
Let $V = \operatorname{span}\{v\}$ be a simple $\Bbbk 1_{\Bbbk^{G_{f}}}$-comodule, i.e., $$\rho(v) = v \otimes ( \sum\limits_{x \in G_f} p_x ).$$ According to Theorem \ref{prop:cqtandcomod}, we have $$ \sum\limits_{j=1}^n  b_{jj}^{(z^{-1}gz)\triangleleft(z^{-1}\triangleright f)}= \sum\limits_{j=1}^n b_{jj}^{z^{-1}gz}$$ for any $ g\in G_f, z\in T_f$.
\end{proof}
An immediate consequence is that, under certain conditions, the coquasitriangular structure explicitly determines the relationships between the simple comodules of $\Bbbk^{G_f}_{\tau_f}$ and $\Bbbk^{G_{f^\prime}}_{\tau_{f^\prime}}$.
\begin{proposition}\label{prop:aiibjj}
Let $\Bbbk^G{}^\tau\#_\sigma\Bbbk F$ be a coquasitriangular Hopf algebra. Suppose that $u\triangleright v=v$ for any $u\in G, v\in F$, that is, $G_v=G.$ For any $f, f^\prime\in F$, if $V = \operatorname{span}\{v^{(1)}, \dots, v^{(m)}\}$ and $W = \operatorname{span}\{w^{(1)}, \dots, w^{(n)}\}$ are simple right comodules over $\Bbbk^{G}_{\tau_f}$ and $\Bbbk^{G}_{\tau_{f^\prime}}$, with coactions $$\rho(v^{(i)})=\sum_{l=1}^m v^{(l)}\otimes (\sum_{g\in G}a_{li}^g p_g),$$ and $$\rho(w^{(j)})=\sum_{k=1}^n w^{(k)}\otimes (\sum_{h\in G}b_{kj}^{h} p_{h}),$$
defined for all $1\leq l, i\leq  m, 1\leq k, j\leq n$, and for all $a_{li}^g, b_{kj}^h\in\k$. Then for any $ g\in G,$ we have
$\sum\limits_{i=1}^m \sum\limits_{j=1}^n a_{ii}^g b_{jj}^{g\triangleleft f}\sigma(g; f, f^\prime)=\sum\limits_{i=1}^m \sum\limits_{j=1}^n a_{ii}^{g\triangleleft f^\prime }b_{jj}^g \sigma(g; f^\prime, f).$
\end{proposition}
\begin{proof}
An argument similar to the one used in the proof of Theorem \ref{prop:OfOf} shows that
\begin{eqnarray*}
&&\chi(\tilde{V})\chi(\tilde{W})\\
&=&\sum\limits_{i=1}^m\sum\limits_{g\in G}\sum\limits_{j=1}^n\sum\limits_{h\in G}
a_{ii}^{g}b_{jj}^{h}\delta_{g\triangleleft f, h}
\sigma(g; f, f^\prime)p_{g}\# f f^{\prime},\\
&=&\sum\limits_{i=1}^m\sum\limits_{g\in G}\sum\limits_{j=1}^n\sum\limits_{h\in G}
a_{ii}^{h}b_{jj}^{g} \delta_{g\triangleleft f^\prime, h}\sigma(g; f^\prime, f)p_{g}\# f^\prime f\\
&=&\chi(\tilde{W})\chi(\tilde{V}).
\end{eqnarray*}
According to Theorem \ref{prop:OfOf}, we have $ff^\prime=f^\prime f.$
It is straightforward to show that  $$\sum\limits_{i=1}^m \sum\limits_{j=1}^n a_{ii}^g b_{jj}^{g\triangleleft f}\sigma(g; f, f^\prime)=\sum\limits_{i=1}^m \sum\limits_{j=1}^n a_{ii}^{g\triangleleft f^\prime }b_{jj}^g \sigma(g; f^\prime, f).$$
\end{proof}

Let $\pi:G\rightarrow G^\prime$ be an epimorphism of groups whose codomain $G^\prime$ is an abelian group. Note that we have a coalgebra embedding from $\k^{G^\prime}$ to $\k^G$. Furthermore, the fact that $G^\prime$ is abelian ensures that all the simple comodules of $\k^{G^\prime}$ are one-dimensional. Based on this observation, we obtain the following corollary.
\begin{corollary}\label{coro:agbg}
Let $\pi: G \rightarrow G^\prime$ be a surjective homomorphism onto an abelian group $G^\prime$. Suppose that $\Bbbk^G\#\Bbbk F$ is a coquasitriangular Hopf algebra and $u\triangleright v=v$ for any $u\in G, v\in F.$
If $V = \operatorname{span}\{v\}$ and $W = \operatorname{span}\{w\}$ are simple right comodules over $\Bbbk^{G^\prime}$, with coactions $$\rho(v)= v\otimes (\sum_{g\in G^\prime}a^g p_g),$$ and $$\rho(w)= w\otimes (\sum_{h\in G^\prime}b^{h} p_{h}),$$
defined for all $a^g, b^h\in\k$. Then for any $f, f^\prime\in F, g\in G^\prime$, we have
$ a^g b^{g\triangleleft f}=  a^{g\triangleleft f^\prime }b^g.$
\end{corollary}
\begin{proof}
Observe that there is a coalgebra embedding of $\k^{G^\prime}$ into $\k^G$. Hence, every simple right $\k^{G^\prime}$-comodule is also a simple right $\k^G$-comodule; in particular, this applies to $V$ and $W$. Applying Proposition \ref{prop:aiibjj}, we can prove this corollary.
\end{proof}
\begin{corollary}\label{ag=agf}
Let $\Bbbk^G\#\Bbbk F$ be a coquasitriangular Hopf algebra and $u\triangleright v=v$ for any $u\in G, v\in F.$ Suppose that $V = \operatorname{span}\{v\}$ is a $1$-dimensional simple right $\k^G$-comodule whose comodule structure is
$$\rho(v)= v\otimes (\sum_{g\in G}a^g p_g).$$ Then $a^g=a^{g\triangleleft f}$ for any $g\in G, f\in F$.
\end{corollary}
\begin{proof}
Note that $1_{\k^G}=\sum\limits_{g\in G}p_g$.  Let $W=\span\{w\}$ be a simple $\k 1_{\k^G}$-comodule. According to Proposition \ref{prop:aiibjj}, we have $a^g=a^{g\triangleleft f}$ for any $g\in G, f\in F.$
\end{proof}

\begin{example}\label{ex:2}
Let $\mathbb{Z}_3=\langle g\mid g^3=1\rangle$. Define group actions $\mathbb{Z}_{3}\xleftarrow{\triangleleft}\mathbb{Z}_{3}\times \mathbb{Z} \xrightarrow{\triangleright}\mathbb{Z}$ on the sets by
$$
g^{i}\triangleright j= j,\;\; g^{i}\triangleleft j=\left\{
\begin{aligned}
g^{i}~,~~~ &\text{if} ~~~  j ~~\text{is even}; \\
g^{3-i},~~~&\text{if} ~~~  j ~~\text{is odd},
\end{aligned}
\right.
$$
for any $0\leq i\leq 2$, $j\in \mathbb{Z}$. The pair of groups $(\mathbb{Z}, \mathbb{Z}_{3})$, together with the mutual group actions $ \mathbb{Z}_{3} \xleftarrow{\triangleleft} \mathbb{Z}_{3} \times \mathbb{Z} \xrightarrow{\triangleright} \mathbb{Z} $, forms a matched pair. Note that the algebra $\k^{\mathbb{Z}_3}$ has only three irreducible corepresentations, denoted $U=\operatorname{span}\{u\}, V=\operatorname{span}\{v\}, W=\operatorname{span}\{w\}$, with their respective comodule structures given by:
\begin{eqnarray*}
\rho(u)&=&u\otimes (p_1+\omega p_g+w^2 p_{g^2}),\\
\rho(v)&=&v\otimes (p_1+\omega^2 p_g+w p_{g^2}),\\
\rho(w)&=&w\otimes (p_1+ p_g+ p_{g^2}),
\end{eqnarray*}
where $\omega$ is a $3$-th primitive root of unit.
An application of Corollary \ref{ag=agf} to $V$ and an arbitrary odd number $j$ establishes that $\k^{\mathbb{Z}_3}\# \mathbb{Z}$ admits no coquasitriangular structure.
\end{example}
\begin{remark}
According to \cite[lemma 5.1]{YLZZ25}, the Hopf algebras constructed in Examples \ref{ex:1} and \ref{ex:2} are in fact group algebras, and as is well-known, a group algebra admits coquasitriangular structures if and only if the underlying group is abelian. Since the two abelian extensions in Examples \ref{ex:1} and \ref{ex:2} are not central, the associated Hopf algebras necessarily noncommutative and thus admit no coquasitriangular structure.
\end{remark}
Suppose that $G$ is a non-abelian group and $u\triangleright v=v$ for any $u\in G, v\in F.$ It can be shown that at this point, $\Bbbk^G\#\Bbbk F$ is a cosemisimple Hopf algebra that is not a group algebra.
For instance, the Hopf algebra $\k^{\mathcal{Q}_8}\# \Bbb{D}_\infty$ presented in Example \ref{ex:q8dinf} is not a group algebra.
In what follows, we introduce a new example of a cosemisimple Hopf algebra that is not a group algebra and admits no coquasitriangular structure.
\begin{example}
Let $\mathcal{Q}_8$ be the quaternion group $\langle r, s\mid r^4=1, r^2=s^2, sr=r^{-1}s\rangle$. Define group actions $\mathcal{Q}_{8}\xleftarrow{\triangleleft}\mathcal{Q}_8\times \mathbb{Z} \xrightarrow{\triangleright}\mathbb{Z}$ on the sets by
\begin{eqnarray*}
u\triangleright j= j,\;\; r\triangleleft j=\left\{
\begin{aligned}
r,~~~ &\text{if} ~~~  j ~~\text{is even}; \\
s,~~~&\text{if} ~~~  j ~~\text{is odd},
\end{aligned}
\right.
\;\;
s\triangleleft j=\left\{
\begin{aligned}
s,~~~ &\text{if} ~~~  j ~~\text{is even}; \\
r,~~~&\text{if} ~~~  j ~~\text{is odd},
\end{aligned}
\right.\;\;
\end{eqnarray*}
for any $u\in \mathcal{Q}_8$, $j\in \mathbb{Z}$. The mutual group actions $ \mathcal{Q}_8 \xleftarrow{\triangleleft} \mathcal{Q}_{8} \times \mathbb{Z} \xrightarrow{\triangleright} \mathbb{Z} $ define a matched pair structure between $\mathbb{Z}$ and $\mathcal{Q}_8$.
Let $\mathcal{K}_4$ be the Kelin four-group $\langle a, b\mid a^2=b^2=1, ab=ba\rangle$. We have a coalgebra embedding from $\k\mathcal{K}_4$ to $\k^{\mathcal{Q}_8}$. Now consider the group-like element
$X=\sum\limits_{k=0}^3\sum\limits_{l=0}^1(-1)^kp_{r^ks^l}$
in $\k\mathcal{K}_4.$ We have the following comodule $V=\operatorname{span}\{v\}$, with comodule structures given by
$$\rho(v)=v\otimes X.$$
 Let $j$ be an odd number, and we consider the situation where $a^{r}=-1$ and $ a^{r\triangleleft j }=a^s=1$. It follows from Corollary \ref{ag=agf}
that $\k^{\mathcal{Q}_8}\# \k\Bbb{Z}$ admits no coquasitriangular structure.
\end{example}
Recall that a \textit{quasitriangular Hopf algebra} (\cite[Definition 10.1.5]{Mon93}) is a pair $(H, R)$, where $H$ is a Hopf algebra and $R=\sum R^{(1)}\otimes R^{(2)}$ is an invertible element in $H\otimes H$ such that
\begin{itemize}
\item[(1)]$(\Delta\otimes \id)(R)=R_{13}R_{23},\;\;(\id\otimes \Delta(R))=R_{13}R_{12};$
\item[(2)]$\Delta^{op}(h) R=R\Delta(h)\;\text{for any } h\in H. $
\end{itemize}
Here by definition $R_{12}=\sum R^{(1)}\otimes R^{(2)} \otimes 1, R_{13}=\sum R^{(1)}\otimes 1\otimes R^{(2)}$ and $R_{23}=\sum 1\otimes R^{(1)}\otimes R^{(2)}$.

For any $g, g^\prime\in G$, define
$O^{\prime}_g=\{g\triangleleft f\mid f\in F\},$ and
$ O^{\prime}_gO^{\prime}_{g^\prime}=\{xy\mid x\in O^{\prime}_g, y\in O^{\prime}_{g^\prime}\}.$
In the case of finite groups $F$ and $G$, a necessary condition for $\Bbbk^G{}^\tau\#_{\sigma}\Bbbk F$ to be quasitriangular is derived.
\begin{proposition}\label{prop:QT=OgOg}
Let $F, G$ be finite groups and $\Bbbk^G{}^\tau\#_{\sigma}\Bbbk F$ be a quasitriangular Hopf algebra.
Then $O^{\prime}_gO^{\prime}_{g^\prime}=O^{\prime}_{g^\prime}O^{\prime}_g$ for any $g, g^\prime \in G.$
\end{proposition}
\begin{proof}
The matched pair condition for $(F, G, \triangleleft, \triangleright)$ directly implies that $(G^{\mathrm{op}}, F^{\mathrm{op}}, \triangleright, \triangleleft)$ is also a matched pair.
It is straightforward to show that $$(\k^G{}^\tau\#_\sigma \k F)^*\cong (\k^{F^{op}}{}^{\tilde{\sigma}}\#_{\tilde{\tau}}\k G^{op})^{op\;cop}$$ as Hopf algebras, where the cocycles $\tilde{\sigma}$ and $\tilde{\tau}$ are given, for all $g, g' \in G$ and $f, f' \in F$, by $\tilde{\sigma}(f, f^\prime ;g):=\sigma(g; f^\prime ,f)$ and $\tilde{\tau}(f; g, g^\prime):=\tau(g^\prime, g;f)$. It is clear that $$(\k^G{}^\tau\#_\sigma \k F)^*\cong \k^{F^{op}}{}^{\tilde{\sigma}}\#_{\tilde{\tau}}\k G^{op}$$ as Hopf algebras and thus $\k^{F^{op}}{}^{\tilde{\sigma}}\#_{\tilde{\tau}}\k G^{op}$ is a coquasitriangular Hopf algebra. According to Theorem \ref{prop:OfOf}, we have $$O^{\prime}_gO^{\prime}_{g^\prime}=O^{\prime}_{g^\prime}O^{\prime}_g$$ for any $g, g^\prime \in G.$
\end{proof}

\begin{remark}
When $F$ and $G$ are finite groups, all results in this section hold in a corresponding quasitriangular form. In other words, this allows us to derive corresponding results in the case where $\Bbbk^G{}^\tau\#_{\sigma}\Bbbk F$
 is a quasitriangular Hopf algebra.
\end{remark}

\begin{example}
Let $\mathcal{S}_3=\{(1), (1\;2), (1\;3), (2\;3), (1\;2\;3), (1\;3\;2) \}$ be the symmetric group of degree $3$. Define group actions $\mathbb{Z}_{2}\xleftarrow{\triangleleft}\mathbb{Z}_{2}\times \mathcal{S}_3 \xrightarrow{\triangleright}\mathcal{S}_3$ on the sets by
$$
x\triangleleft \nu =x,\;\; 1\triangleright \nu=\nu,\;\; g\triangleright \nu= (1\;2)\nu(1\;2),
$$
for any $x\in\mathbb{Z}_2, \nu\in \mathcal{S}_3.$ The pair of groups $(\mathcal{S}_3, \mathbb{Z}_{2})$, together with the mutual group actions forms a matched pair. Thus, $\k^{\mathbb{Z}_2}\# \mathcal{S}_3$ is a Hopf algebra, which is precisely the dual Hopf algebra of $A_+$ defined in \cite[Definition 4.1]{Fuk97}.
It is well-known that the Hopf algebra $A_+$ is self-dual and admits a quasitriangular structure; consequently, it also admits a coquasitriangular structure.
 This example illustrates that the existence of a coquasitriangular structure on
$\Bbbk^G{}^\tau\#_{\sigma}\Bbbk F$ is possible even for a non-abelian group
$F$, as long as the condition $O_f O_{f^\prime} = O_{f^\prime} O_f$ (for any $f, f^\prime \in F$) from Theorem \ref{prop:OfOf} is satisfied.
\end{example}
\section{Coquasitriangular structures on $\Bbbk^G{}^\tau\#_{\sigma}\Bbbk F$}\label{Section 4}
In this section, let $G$ be a finite group and $F$ an arbitrary group. We establish a characterization of the coquasitriangular structures on $\Bbbk^G{}^\tau\#_{\sigma}\Bbbk F$.

First, based on the definition of a coquasitriangular structure, we state the following proposition.
\begin{proposition}\label{prop:CQTstrcuture}
Let $H=\Bbbk^G{}^\tau\#_{\sigma}\Bbbk F$ and $R: H\otimes H\rightarrow \Bbbk$ is a bilinear form on $H$. Then $(H, R)$ is coquasitriangular if and only if the following conditions are fulfilled:
  \begin{eqnarray*}
  (CQT0):&&\sum_{x\in G}R(p_x\#1, p_{g}\# f)=\sum_{x\in G}R(p_{g}\# f, p_x\#1)=\delta_{g, 1},\\
 (CQT1):&& \delta_{h\triangleleft f^\prime,l}\sigma(h; f^\prime ,f^{\prime\prime})R(p_g\# f, p_h\# f^\prime f^{\prime\prime})\\
 &=&\sum\limits_{x\in G}\tau(gx^{-1}, x; f)R(p_{gx^{-1}}\#(x\triangleright f), p_l\# f^{\prime\prime})R(p_x\# f, p_h\# f^\prime),\\
  (CQT2):&&\delta_{g\triangleleft f, h}\sigma(g; f, f^\prime)R(p_g\# ff^\prime, p_l\# f^{\prime\prime})\\
  &=&\sum_{x\in G}\tau(lx^{-1}, x;f^{\prime\prime})R(p_g\# f, p_{lx^{-1}}\#(x\triangleright f^{\prime\prime})) R(p_h\#f^\prime, p_x \# f^{\prime\prime}),\\
  (CQT3):&&\tau(gl^{-1}, l; f)\tau(h(l\triangleleft f)^{-1}, l\triangleleft f;f^\prime) R(p_{gl^{-1}}\# (l\triangleright f),p_{h(l\triangleleft f)^{-1}}\#((l\triangleleft f)\triangleright f^{\prime}))\\
&&\sigma(l; f, f^\prime) ff^\prime\\
&=&\tau((h\triangleleft f^\prime)((l^{-1}h)\triangleleft f^\prime)^{-1},  ((l^{-1}h)\triangleleft f^\prime)(h\triangleleft f^\prime)^{-1} g; f)\tau(l, l^{-1}h; f^{\prime})
\\&&R(p_{((l^{-1}h)\triangleleft f^\prime)(h\triangleleft f^\prime)^{-1}g}\# f, p_{l^{-1}h}\# f^\prime)\\
&&\sigma(l; (l^{-1}h)\triangleright f^\prime, (((l^{-1}h)\triangleleft f^\prime)(h\triangleleft f^\prime)^{-1} g)\triangleright f)\\
&&((l^{-1}h)\triangleright f^\prime)((((l^{-1}h)\triangleleft f^\prime)(h\triangleleft f^\prime)^{-1} g)\triangleright f)
  \end{eqnarray*}
  for any $g, h, l\in G$ and $f, f^\prime, f^{\prime\prime}\in F.$
  Moreover, $H$ is cotriangular provided that
   \begin{eqnarray*}
   (CQT4):&&\sum_{x,y\in G} \tau(gx^{-1}, x; f)\tau(hy^{-1}, y; f^{\prime})R(p_{gx^{-1}}\# (x\triangleright f), p_{hy^{-1}}\#(y\triangleright f^{\prime}))\;\;\;\;\;\;\;\;\;\;\;\;\\
   &&R(p_y\# f^\prime, p_x\# f)\\
   &=&\delta_{g,1_G}\delta_{h,1_G}
   \end{eqnarray*}
is also fulfilled for any $g,h\in G, f, f^\prime \in F.$
\end{proposition}
\begin{proof}
Through direct computation, we find that (\ref{CQT1}) holds if and only if both (CQT1) and (CQT2) hold, and (\ref{CQT4}) holds if and only if (CQT4) holds. In order to prove the equivalence between (\ref{CQT3}) and (CQT3), we start by first computing the left-hand side of (\ref{CQT3}). For any $p_g\#f,p_h\#f^\prime \in H$, we have
\begin{eqnarray*}
LHS&=&\sum_{x,y\in G}\tau(gx^{-1}, x; f)\tau(hy^{-1}, y; f^{\prime})R(p_{gx^{-1}}\# (x\triangleright f), p_{hy^{-1}}\#(y\triangleright f^{\prime}))\\
&&(p_x\# f)(p_y\# f^\prime)\\
&=&\sum_{x,y\in G}\tau(gx^{-1}, x; f)\tau(hy^{-1}, y; f^{\prime})R(p_{gx^{-1}}\# (x\triangleright f), p_{hy^{-1}}\#(y\triangleright f^{\prime}))\\
&&\delta_{x\triangleleft f, y}\sigma(x; f, f^\prime) p_x\# ff^\prime\\
&=&\sum_{x\in G}\tau(gx^{-1}, x; f)\tau(h(x\triangleleft f)^{-1}, x\triangleleft f;f^\prime) R(p_{gx^{-1}}\# (x\triangleright f),p_{h(x\triangleleft f)^{-1}}\#((x\triangleleft f)\triangleright f^{\prime}))\\
&&\sigma(x; f, f^\prime) p_x\# ff^\prime.
  \end{eqnarray*}
Now we calculate the right-hand side of (\ref{CQT3}). It follows that
\begin{eqnarray*}
RHS&=&\sum_{x,y\in G}\tau(gx^{-1}, x; f)\tau(hy^{-1}, y; f^{\prime})R(p_x\# f, p_y\# f^\prime)\\
&&(p_{hy^{-1}}\#(y\triangleright f^{\prime}))(p_{gx^{-1}}\# (x\triangleright f))\\
&=&\sum_{x,y\in G}\tau(gx^{-1}, x; f)\tau(hy^{-1}, y; f^{\prime})R(p_x\# f, p_y\# f^\prime)\delta_{(hy^{-1})\triangleleft (y\triangleright f^\prime), gx^{-1}}\\
&&\sigma(hy^{-1}; y\triangleright f^\prime, x\triangleright f)
p_{hy^{-1}}\#(y\triangleright f^\prime)(x\triangleright f).
 \end{eqnarray*}
As established by Remark \ref{lem:actioninverse}, we have
 \begin{eqnarray*}
 (hy^{-1})\triangleleft (y\triangleright f^\prime)&=&(h\triangleleft (y^{-1}\triangleright(y\triangleright f^\prime)))(y^{-1}\triangleleft (y\triangleright f^\prime))\\
 &=& (h\triangleleft f^\prime)(y\triangleleft f^\prime)^{-1}.
 \end{eqnarray*}
 Thus we obtain
 \begin{eqnarray*}
RHS&=& \sum_{y\in G}\tau( (h\triangleleft f^\prime)(y\triangleleft f^\prime)^{-1},  (y\triangleleft f^\prime)(h\triangleleft f^\prime)^{-1} g; f)\tau(hy^{-1}, y; f^{\prime})
\\&&R(p_{(y\triangleleft f^\prime)(h\triangleleft f^\prime)^{-1}g}\# f, p_y\# f^\prime)
\sigma(hy^{-1}; y\triangleright f^\prime, ((y\triangleleft f^\prime)(h\triangleleft f^\prime)^{-1} g)\triangleright f)\\
&&p_{hy^{-1}}\#(y\triangleright f^\prime)(((y\triangleleft f^\prime)(h\triangleleft f^\prime)^{-1} g)\triangleright f).
 \end{eqnarray*}
 Observe that the set $\{p_u\# v\mid u\in G, v\in F\}$ is linearly independent.
Consequently, the equality (\ref{CQT3}) holds if and only if
\begin{eqnarray*}
&&\tau(gx^{-1}, x; f)\tau(h(x\triangleleft f)^{-1}, x\triangleleft f;f^\prime) R(p_{gx^{-1}}\# (x\triangleright f),p_{h(x\triangleleft f)^{-1}}\#((x\triangleleft f)\triangleright f^{\prime}))\sigma(x; f, f^\prime) \\&&p_x\#ff^\prime\\
&=&\tau( (h\triangleleft f^\prime)((x^{-1}h)\triangleleft f^\prime)^{-1},  ((x^{-1}h)\triangleleft f^\prime)(h\triangleleft f^\prime)^{-1} g; f)\tau(x, x^{-1}h; f^{\prime})
\\&&R(p_{((x^{-1}h)\triangleleft f^\prime)(h\triangleleft f^\prime)^{-1}g}\# f, p_{x^{-1}h}\# f^\prime)\sigma(x; (x^{-1}h)\triangleright f^\prime, (((x^{-1}h)\triangleleft f^\prime)(h\triangleleft f^\prime)^{-1} g)\triangleright f)\\
&&p_x\#((x^{-1}h)\triangleright f^\prime)((((x^{-1}h)\triangleleft f^\prime)(h\triangleleft f^\prime)^{-1} g)\triangleright f)
\end{eqnarray*}
holds for any $x\in G.$ This implies that (\ref{CQT3}) is equivalent to (CQT3).
Moreover, when $(CQT0)$ and $(CQT1)$ hold, one can show that
$R$ is convolution invertible and that its inverse satisfies $R^{-1}(a, b)=R(S(a), b)$ for all $a, b\in H$. This means that if equations (CQT0)-(CQT3) hold, then
$(H, R)$ is a coquasitriangular Hopf algebra. Conversely, assume $(H, R)$ is coquasitriangular. Then (\ref{CQT1}) yields, for any $a, b\in H,$$$R(1_H, a)=R(1_H, a_{(1)})R(1_H, a_{(2)})$$ and
$$R(1_H, ab)=R(1_H, a)R(1_H, b).$$
This means that $R(1_H, \cdot)$ is both an algebra homomorphism and convolution idempotent, forcing $R(1_H, a)=\epsilon (a)$. Furthermore, by a symmetric argument, one finds $R( a,1_H)=\epsilon (a)$. Hence, condition (CQT0) is satisfied.
\end{proof}

This leads naturally to the following corollarys.
\begin{corollary}\label{coro:Rgf=0}
Let $(\Bbbk^G{}^\tau\#_{\sigma}\Bbbk F, R)$ be a coquasitriangular Hopf algebra, with elements $g, h \in G, f, f^\prime\in F.$ If $ff^\prime \neq (h\triangleright f^\prime)(g\triangleright f)$, then $R(p_g\# f, p_h\#f^\prime)=0.$
\end{corollary}
\begin{proof}
By setting $l=1$ in (CQT3), we obtain
$$R(p_g\# f, p_h\#f^\prime) ff^\prime=R(p_g\# f, p_h\#f^\prime)(h\triangleright f^\prime)(g\triangleright f).$$This implies that $R(p_g\# f, p_h\# f^\prime)$ vanishes unless $ff^\prime = (h\triangleright f^\prime)(g\triangleright f)$.
\end{proof}
\begin{corollary}\label{coro:FabelianRgf=0}
Let $F$ be an abelian group and consider the coquasitriangular Hopf algebra $(\Bbbk^G {}^\tau \#_{\sigma} \Bbbk F, R)$. For $g, h \in G$ and $f, f^\prime \in F$,  if either $g \in G_f$ and $h \notin G_{f^\prime}$, or $g \notin G_f$ and $h \in G_{f^\prime}$, then $R(p_g\# f, p_h\#f^\prime)=0.$
\end{corollary}
\begin{proof}
We can complete the proof by an appeal to Corollary \ref{coro:Rgf=0}.
\end{proof}
\begin{corollary}\label{coro:ginGfR=}
Let $(\Bbbk^G{}^\tau\#_{\sigma}\Bbbk F, R)$ be a coquasitriangular Hopf algebra, with elements $g, h \in G, f\in F.$
If $g\notin G_f$, then $R(p_{g}\# (g^{-1}\triangleright f), p_{h}\#1)=R(p_{g}\#f, p_{h}\#1)=0,$
and
$
R(p_{h}\# 1, p_{g}\#f)=R(p_{h}\#1, p_{g}\# (g^{-1}\triangleright f))=0.
$
\end{corollary}
\begin{proof}
Substituting
$f^\prime=1$ into (CQT3) gives us
\begin{eqnarray*}
&&\tau(gl^{-1}, l; f) R(p_{gl^{-1}}\# (l\triangleright f),p_{h(l\triangleleft f)^{-1}}\#1) f\\
&=&\tau(l, l^{-1}g; f)
R(p_{l^{-1}g}\# f, p_{l^{-1}h}\# 1)
(l^{-1}g\triangleright f).
  \end{eqnarray*}
First, set
$g=1$ in the above equation. This establishes the first statement. Next, by setting $f = 1$ and proceeding analogously, we can derive another result.
\end{proof}

Let $E$ be a group. Recall that a \textit{bicharacter} on $E$ is
a function
$
\beta: E \times E \to \k
$
satisfying the following conditions for all $ x, y, z \in E :$
\begin{eqnarray*}
   && \beta(xy, z)= \beta(x, z) \beta(y, z),\\
   && \beta(x, yz) = \beta(x, y) \beta(x, z),\\
    &&\beta(x,1)=\beta(1,x)=1.
    \end{eqnarray*}
For an abelian group $E$, the bicharacters on $E$ are in bijection with the coquasitriangular structures on the group algebra $\k E$. In particular, the formula $R(x, y): = \varepsilon(x)\varepsilon(y)$ yields the standard coquasitriangular structure on $\k E$.

\begin{proposition}\label{prop:bicha}
Let $(\Bbbk^G{}^\tau\#_{\sigma}\Bbbk F, R)$ be a coquasitriangular Hopf algebra such that the extension $\k^G\rightarrow \k^G{}^\tau\#_\sigma \k F\rightarrow \k F $ is central. Denote
$S:=\{f\in F\mid g\triangleright f=f \;\text{for any } g\in G\}.$ Suppose for any $f\in F$, all the simple right $\k^{G_f}_{\tau_f}$-comodules are $1$-dimensional. Then
\begin{itemize}
  \item [(1)]$\k^{G}{}^\tau\#_\sigma \k S$ is an abelian group algebra;
  \item [(2)]$\sigma(g; f, f^\prime) = \sigma(g; f^\prime, f)$ holds for all $g\in G, f, f^\prime \in S$;
  \item [(3)]$S$ is an abelian group;
  \item [(4)]The restriction of $R$ to $(\k^{G}{}^\tau\#_\sigma \k S)\otimes (\k^{G}{}^\tau\#_\sigma \k S)$ is a bicharacter on $\k^{G}{}^\tau\#_\sigma \k S$.
\end{itemize}
\end{proposition}
\begin{proof}
It is clear that $(S, G)$ forms a matched pair, making $\k^{G}{}^\tau\#_{\sigma} \k S$ a Hopf algebra.
According to Lemma \ref{lem:simplecomod} and \cite[Lemma 5.1]{YLZZ25}, this Hopf algebra is isomorphic to a group algebra. Since a group algebra admits a coquasitriangular structure if and only if the underlying group is abelian, it follows that $\k^{G}{}^\tau\#_{\sigma} \k S$ must be an abelian group algebra. Now, for any $g\in G, f, f' \in S$, comparing the products
 $$(p_g\#f)(p_g\#f') = \sigma(g; f, f')ff'$$
 and
 $$(p_g\#f')(p_g\# f) = \sigma(g; f', f)ff'$$
 in this abelian group algebra, we find that $$\sigma(g; f, f') = \sigma(g; f', f).$$
 Besides, since
 $$(p_1\#f)(p_1\#f') = (p_1\#f')(p_1\# f)$$
 in this abelian group algebra, it follows that $ff^\prime= f^\prime f$ for any $f, f^\prime \in S.$
\end{proof}

\section{Coquasitriangular structures on $\k^{\mathbb{Z}_2}{}^\tau\#_\sigma\k F$}\label{Section5}
In the following part, we focus on the case where $G$ is the cyclic group of order 2, $\mathbb{Z}_2 = \langle g \mid g^2 = 1 \rangle$, and $F$ is an arbitrary group. We investigate when a coquasitriangular structure exists on $\k^{\mathbb{Z}_2}{}^\tau\#_\sigma\k F$ and provide characterizations of its coquasitriangular structures.

In this case, the action of $F$ on $\mathbb{Z}_2$ is necessarily trivial. Moreover, for any $f\in F,$
we note that $G_f$ is a subgroup of $\mathbb{Z}_2$ for any $f\in F$, which implies that
 $G_f$ is precisely $\{1\}$ or $\mathbb{Z}_2.$

Inspired by \cite{ZL21}, we use the following notations:
\begin{eqnarray*}
S:=\{f\in F\mid g\triangleright f=f\},\\
T:=\{f\in F\mid g\triangleright f\neq f\}.
\end{eqnarray*}
It is clear that $F=S\cup T.$
For any $f\in F$, the isomorphism classes of simple right $\k^{G_f}_{\tau_f}$-comodules are determined as follows:
\begin{itemize}
\item For $f\in S$ (where $G_f=\mathbb{Z}_2$), there are two classes. They are represented by the one-dimensional comodules $U=\span\{u\}$ and $V=\span\{v\}$, parameterized by the coactions:
\begin{eqnarray*}
\rho(u)=u\otimes (p_1+\sqrt{\tau(g,g;f)}p_g),\\
\rho(v)=v\otimes (p_1-\sqrt{\tau(g,g;f)}p_g).
\end{eqnarray*}
\item For $f\in T$ (where $G_f=\k \{1\}$), there is a unique class, represented by the trivial comodule $W=\span\{w\}$ with the coaction
\begin{eqnarray*}
\rho(w)=w\otimes p_1.
\end{eqnarray*}
\end{itemize}

According to Lemma \ref{lem:simplecomod}, every simple right $\k^{\mathbb{Z}_2}{}^\tau\#_\sigma\k F$-comodule can be constructed as a comodule induced from $\k^{G_{f}}_{\tau_{f}}$ for some $f \in F$.
Therefore, the irreducible characters for all simple comodules over $\k^{\mathbb{Z}_2}{}^\tau\#_\sigma\k F$ are given by:
\begin{eqnarray*}
\chi(\tilde{U}_f)&=&p_1\# f+\sqrt{\tau(g,g;f)}p_g\#f\;\text{for any }f\in S,\\
\chi(\tilde{V}_f)&=&p_1\# f-\sqrt{\tau(g,g;f)}p_g\#f\;\text{for any }f\in S,\\
\chi(\tilde{W}_f)&=&p_1\# f+p_1\# (g\triangleright f)\;\text{for any }f\in T.
\end{eqnarray*}
It is clear that $\tilde{W}_f\cong \tilde{W}_{g\triangleright f}$ for any $f\in T.$

\begin{proposition}\label{prop:Grz2}
Let $H=\k^{\mathbb{Z}_2}{}^\tau\#_\sigma\k F$.
Then the Grothendieck ring $\operatorname{Gr}(H\text{-comod})$ is determined by
\begin{eqnarray*}
&\tilde{U}_f\otimes \tilde{U}_{f^\prime}\cong \tilde{V}_{f}\otimes \tilde{V}_{f^\prime}\cong\tilde{U}_{ff^\prime},&\text{for any } f,f^\prime\in S,\\
&\tilde{U}_f\otimes\tilde{V}_{f^\prime}\cong \tilde{V}_f\otimes\tilde{U}_{f^\prime} \cong \tilde{V}_{ff^\prime},&\text{for any } f,f^\prime\in S,\\
&\tilde{U}_f\otimes \tilde{W}_{f^\prime}\cong \tilde{V}_f\otimes \tilde{W}_{f^\prime}\cong \tilde{W}_{ff^\prime},\;\; \tilde{W}_{f^\prime}\otimes \tilde{U}_f\cong \tilde{W}_{f^\prime}\otimes\tilde{V}_f \cong \tilde{W}_{f^\prime f},&\text{for any }f\in S, f^\prime \in T,
\end{eqnarray*}
\begin{eqnarray*}
&&\tilde{W}_f\otimes \tilde{W}_{f^\prime}\cong \left\{
\begin{aligned}
&\tilde{W}_{ff^\prime}\oplus\tilde{W}_{f(g\triangleright f^\prime)},~~~ \text{if} ~~~  ff^\prime, f(g\triangleright f^\prime)\in T \text{ for } f, f^\prime\in T; \\
&\tilde{U}_{ff^\prime}\oplus\tilde{V}_{ff^\prime}\oplus\tilde{W}_{f(g\triangleright f^\prime)},~~~\text{if} ~~~  ff^\prime\in S, f(g\triangleright f^\prime)\in T\text{ for } f, f^\prime\in T;\\
&\tilde{W}_{ff^\prime}\oplus\tilde{U}_{f(g\triangleright f^\prime)} \oplus\tilde{V}_{f(g\triangleright f^\prime)},~~~\text{if} ~~~  ff^\prime\in T, f(g\triangleright f^\prime)\in S\text{ for } f, f^\prime\in T;\\
&\tilde{U}_{ff^\prime} \oplus\tilde{V}_{ff^\prime}\oplus\tilde{U}_{f(g\triangleright f^\prime)} \oplus\tilde{V}_{f(g\triangleright f^\prime)},~~~\text{if} ~~~  ff^\prime, f(g\triangleright f^\prime)\in S\text{ for } f, f^\prime\in T.
\end{aligned}
\right.
\end{eqnarray*}
\end{proposition}
\begin{proof}
 Using \cite[Lemma 5.1]{YLZZ25}, we have $\operatorname{dim}_{\k}(\tilde{U}_f)=\operatorname{dim}_{\k}(\tilde{V}_f)=1$ for any $f\in S$ and $\operatorname{dim}_{\k}(\tilde{W}_f)=2$ for any $f\in T.$ According to Lemma \ref{lem:Gr=character}, for any simple right $H$-comodules $M$ and $N$, the tensor product of $M$ and $N$ is isomorphic to a comodule characterized by the product of characters $\chi(M)\chi(N)$.
For $f, f^\prime\in S$, a direct computation shows that
\begin{eqnarray*}
\chi(\tilde{U}_f)\chi(\tilde{U}_{f^\prime})&=&(p_1\# f+\sqrt{\tau(g,g;f)}p_g\#f)(p_1\# f^\prime+\sqrt{\tau(g,g;f^\prime)}p_g\#f^\prime)\\
&=&p_1\# ff^\prime +\sqrt{\tau(g,g;f)\tau(g,g;f^\prime)}\sigma(g;f, f^\prime)p_g\# ff^\prime.
\end{eqnarray*}
The compatibility between $\sigma$ and $\tau$ implies that
$$
\tau(g, g;ff^\prime)=(\sigma(g;f,f^\prime))^2\tau(g,g;f)\tau(g,g;f^\prime).
$$
It follows that
\begin{eqnarray*}
\chi(\tilde{U}_f)\chi(\tilde{U}_{f^\prime})=\chi(\tilde{U}_{ff^\prime}).
\end{eqnarray*}
A similar argument shows that
\begin{eqnarray*}
\chi(\tilde{V}_f)\chi(\tilde{U}_{f^\prime})=\chi(\tilde{U}_f)\chi(\tilde{V}_{f^\prime})=\chi(\tilde{V}_{ff^\prime}),\;\; \chi(\tilde{V}_f)\chi(\tilde{V}_{f^\prime})=\chi(\tilde{U}_{ff^\prime}),
\end{eqnarray*}
where
$f, f^\prime \in F.$
For $f, f^\prime\in T,$ a direct calculation yields
\begin{eqnarray*}
\chi(\tilde{W}_f)\chi(\tilde{W}_{f^\prime})&=&(p_1\# f+p_1\# (g\triangleright f))(p_1\# f^\prime+p_1\# (g\triangleright f^\prime))\\
&=&p_1\#ff^\prime+p_1\#(g\triangleright f)(g\triangleright f^\prime)+p_1\#f(g\triangleright f^\prime)+p_1\#(g\triangleright f)f^\prime.
\end{eqnarray*}
Note that $$g\triangleright (ab)=(g\triangleright a)(g\triangleright b)$$ for any $a, b\in F$.
It follows that
\begin{eqnarray*}
\chi(\tilde{W}_f)\chi(\tilde{W}_{f^\prime})=\left\{
\begin{aligned}
&\chi(\tilde{W}_{ff^\prime})+\chi(\tilde{W}_{f(g\triangleright f^\prime)}),~~~ \text{if} ~~~  ff^\prime, f(g\triangleright f^\prime)\in T; \\
&\chi(\tilde{U}_{ff^\prime}) +\chi(\tilde{V}_{ff^\prime})+\chi(\tilde{W}_{f(g\triangleright f^\prime)}),~~~\text{if} ~~~  ff^\prime\in S, f(g\triangleright f^\prime)\in T;\\
&\chi(\tilde{W}_{ff^\prime})+\chi(\tilde{U}_{f(g\triangleright f^\prime)}) +\chi(\tilde{V}_{f(g\triangleright f^\prime)}),~~~\text{if} ~~~  ff^\prime\in T, f(g\triangleright f^\prime)\in S;\\
&\chi(\tilde{U}_{ff^\prime}) +\chi(\tilde{V}_{ff^\prime})+\chi(\tilde{U}_{f(g\triangleright f^\prime)}) +\chi(\tilde{V}_{f(g\triangleright f^\prime)}),~~~\text{if} ~~~  ff^\prime, f(g\triangleright f^\prime)\in S.
\end{aligned}
\right.
\end{eqnarray*}
Moreover, for any $f\in S, f^\prime \in T$, it is strightforword to show that
\begin{eqnarray*}
\chi(\tilde{U}_f)\chi(\tilde{W}_{f^\prime})=\chi(\tilde{V}_f)\chi(\tilde{W}_{f^\prime})=\chi(\tilde{W}_{ff^\prime}),\\
\chi(\tilde{W}_{f^\prime})\chi(\tilde{U}_f)=\chi(\tilde{W}_{f^\prime})\chi(\tilde{V}_f)=\chi(\tilde{W}_{f^\prime f}).
\end{eqnarray*}
\end{proof}

The following corollary is an immediate consequence of the preceding proposition.
\begin{corollary}\label{coro:grz2commu}
Let $H=\k^{\mathbb{Z}_2}{}^\tau\#_\sigma\k F$. If the Grothendieck ring $\operatorname{Gr}(H\text{-comod})$ is commutative, then $S$ is an abelian group.
\end{corollary}
\begin{example}\emph{(}\cite[Example 3.6]{YLZZ25}\emph{)}
Let $\mathbb{Z}_2=\{g\mid g^2=1\}$. Define group actions $\mathbb{Z}_2\xleftarrow{\triangleleft}\mathbb{Z}_2\times \mathbb{Z} \xrightarrow{\triangleright}\mathbb{Z}$ on the sets by
$$
1\triangleleft i=1,\;\; g\triangleleft i=g,\;\;1\triangleright i=i,\;\;g\triangleright i=-i,
$$
for any $i\in \mathbb{Z}$. Clearly, $(\mathbb{Z}, \mathbb{Z}_2)$ together with group actions $\mathbb{Z}_2\xleftarrow{\triangleleft}\mathbb{Z}_2\times \mathbb{Z} \xrightarrow{\triangleright}\mathbb{Z}$ on the sets is a matched pair and thus $\k^{\mathbb{Z}_2}\#\k\mathbb{Z}$ forms a Hopf algebra. In fact, $\k^{\mathbb{Z}_2}\#\k\mathbb{Z}$ is exactly the coradical of $H(e_{\pm 1}, f_{\pm 1}, u, v)$ in \cite[Definition 5.1]{YL26}. The commutativity of
$\k^{\mathbb{Z}_2}\#\k\mathbb{Z}$ implies that it admits a coquasitriangular structure.
\end{example}

Notice that for this case, Proposition \ref{prop:bicha} can be expressed as follows.
\begin{corollary}\label{lem:R(gS)}
Suppose $(\k^{\mathbb{Z}_2}{}^\tau\#_{\sigma}\k F, R)$ is a coquasitriangular Hopf algebra. Then $\k^{\mathbb{Z}_2}{}^\tau\#_\sigma \k S$ is an abelian group algebra, and $\sigma(g; f, f^\prime) = \sigma(g; f^\prime, f)$ holds for all $f, f^\prime \in S$. Furthermore, the restriction of $R$ to $(\ k^{\mathbb{Z}_2}{}^\tau\#_\sigma \k S)\otimes (\k^{\mathbb{Z}_2}{}^\tau\#_\sigma \k S)$ is a bicharacter on $\k^{\mathbb{Z}_2}{}^\tau\#_\sigma \k S$.
\end{corollary}
Next, we present the possible values of the coquasitriangular structure
$R$ of the coquasitriangular Hopf algebra $(\k^{\mathbb{Z}_2}{}^\tau\#_\sigma\k F, R)$
at certain special elements.
\begin{lemma}\label{lem:R11}
Let $(\k^{\mathbb{Z}_2}{}^\tau\#_\sigma\k F, R)$ be a coquasitriangular Hopf algebra. Then
one of the following two cases must hold:
\begin{itemize}
\item[(1)]
$
R(p_1\#1, p_1\#1)=R(p_1\#1, p_g\#1)=R(p_g\#1, p_1\#1)=\frac{1}{2},\;\;R(p_g\#1, p_g\# 1)=-\frac{1}{2};
$
\item[(2)]$R(p_1\#1, p_1\#1)=1,$ $R(p_1\#1, p_g\#1)=R(p_g\#1, p_1\#1)=R(p_g\#1, p_g\# 1)=0$.
\end{itemize}
\end{lemma}
\begin{proof}
Setting $f=1$ in (CQT0), we find that for $x=1$ or $g,$ the following holds:
\begin{eqnarray*}
R(p_1\#1, p_x\#1)+R(p_g\#1, p_x\#1)
=R(p_x\#1, p_1\#1)+ R(p_x\#1, p_g\#1)=\delta_{x, 1}.
\end{eqnarray*}
Suppose $R(p_1\#1, p_g\#1)=k$ for some $k\in\k$. It is straightforward to show that
$$R(p_1\#1, p_1\#1)=1-k,\;\;R(p_g\#1, p_1\#1)=k,\;\;R(p_g\#1, p_g\#1)=-k.$$
With the substitutions $f=f^\prime=f^{\prime\prime}=1$ and $g=l=h=1$ in (CQT1), we have
\begin{eqnarray*}
R(p_1\#1,p_1\#1)=R(p_{1}\#1, p_1\#1)R(p_1\#1, p_1\#1)+R(p_{g}\#1, p_1\#1)R(p_g\#1, p_1\#1).
\end{eqnarray*}
This leads to the equation
$$1-k=(1-k)^2+k^2.$$
Solving it gives $k=0$ or $k=\frac{1}{2}$.
\end{proof}

\begin{corollary}
Let $(\k^{\mathbb{Z}_2}{}^\tau\#_\sigma\k F, R)$ be a coquasitriangular Hopf algebra, $x\in \mathbb{Z}_2$ and $f\in F$.
\begin{itemize}
\item[(1)]If $R(p_1\#1, p_g\#1)=\frac{1}{2}$, then
$
R(p_x\#1, p_1\#f)=R(p_1\# f, p_x\#1)=\frac{1}{2}
$;
\item[(2)]If $R(p_1\#1, p_g\#1)=0$, then
$
R(p_x\#1, p_g\#f)=R(p_g\# f, p_x\#1)=0.
$
\end{itemize}
\end{corollary}
\begin{proof}
With
$f=f^\prime =1, h=l, g=1$
 in (CQT1), this gives
 \begin{eqnarray*}
R(p_1\#1, p_x\#f)=R(p_1\#1, p_x\#f)R(p_1\#1, p_x\#1)+R(p_g\#1, p_x\#f)R(p_g\#1, p_x\#1)
 \end{eqnarray*}
 for any $x\in \mathbb{Z}_2.$
Suppose $R(p_1\#1, p_g\#1)=k\in\k$. It follows from (CQT0) and the proof of Lemma \ref{lem:R11} that
\begin{eqnarray*}
\left\{
\begin{aligned}
kR(p_g\#1, p_1\#f)-kR(p_1\#1, p_1\#f)=0, \\
R(p_g\#1, p_1\#f)+R(p_1\#1, p_1\#f)=1,
\end{aligned}
\right.
\end{eqnarray*}
and
\begin{eqnarray*}
\left\{
\begin{aligned}
(1-k)R(p_1\#1, p_g\# f)+kR(p_g\#1, p_g\#f)=0,\\
R(p_1\#1, p_g\# f)+R(p_g\#1, p_g\#f)=0.
\end{aligned}
\right.
\end{eqnarray*}
Substituting $k=\frac{1}{2}$ and $k=0$ into the above system yields the respective results: for $k=\frac{1}{2}$, $$R(p_g\#1, p_1\#f)=R(p_1\#1, p_1\#f)=\frac{1}{2};$$ for $k=0$, $$R(p_1\#1, p_g\# f)=R(p_g\#1, p_g\#f)=0.$$ Similarly, setting $f=f^{\prime\prime}= 1, g=h, l=1$ in (CQT2) and following the same procedure, we can prove the remaining results.
\end{proof}

Furthermore, through some calculations, we can obtain some equations for $R.$
Setting $h=l, f^\prime=1$ in (CQT1), we have
\begin{eqnarray*}
R(p_x\#f, p_h\#f^{\prime\prime})&=&R(p_x\#f, p_h\# f^{\prime\prime})R(p_1\#f, p_h\#1)\\
&&+\tau(xg,g;f)R(p_{xg}\#(g\triangleright f), p_h\# f^{\prime\prime})R(p_g\#f, p_h\#1).
\end{eqnarray*}
Using similar arguments for (CQT2), we find that
\begin{eqnarray*}
R(p_h\# f^\prime, p_l\# f^{\prime\prime})&=& R(p_h\# 1, p_l\# f^{\prime\prime})R(p_h\# f^\prime, p_1\# f^{\prime\prime})\\
&&+\tau(lg, g;f^{\prime\prime})R(p_h\# 1, p_{lg}\#(g\triangleright f^{\prime\prime})) R(p_h\# f^{\prime}, p_g\# f^{\prime\prime}).
\end{eqnarray*}
If $R(p_1\# 1, p_g\#1)=\frac{1}{2}$, then we have
\begin{eqnarray*}
R(p_x\# f, p_h\# f^\prime)&=&2\tau(xg,g;f)R(p_{xg}\#(g\triangleright f), p_h\# f^{\prime})R(p_g\#f, p_h\#1),\\
R(p_x\#f, p_1\#f^\prime)&=&2\tau(g, g;f^{\prime})R(p_x\# 1, p_{g}\#(g\triangleright f^{\prime})) R(p_x\# f, p_g\# f^{\prime}),\\
R(p_x\#f, p_g\# f^\prime)&=&2R(p_x\#1, p_g\#f^\prime) R(p_x\# f, p_1\# f^\prime),
\end{eqnarray*}
where $x, h\in G, f, f^\prime \in F.$
This means that
\begin{eqnarray*}
R(p_x\# f, p_h\# f^\prime)&=& 4\tau(xg, g; f) \tau(x, g; g\triangleright f)R(p_x\#f, p_h\#f^{\prime})R(p_g\#(g\triangleright f), p_h\#1) \\
&&R(p_g\#f, p_h\#1)\\
&=& 4\tau(g,g;f)R(p_x\#f, p_h\#f^{\prime})
R(p_g\#(g\triangleright f), p_h\#1)R(p_g\#f, p_h\#1).
\end{eqnarray*}
If $R(p_1\# 1, p_g\# 1)=0,$ then we obatin
\begin{eqnarray*}
R(p_x\# f, p_h\# f^\prime)&=& R(p_x\# f, p_h\#f^\prime)R(p_1\# f, p_h\#1),\\
R(p_x\# f, p_1\# f^\prime)&=&R(p_x\#1, p_1\#f^\prime) R(p_x\# f, p_1\#f^\prime ),\\
R(p_x\# f, p_g\#f^\prime)&=&R(p_x\# 1, p_1\#(g\triangleright f^\prime))R(p_x\#f, p_g\# f^\prime).
\end{eqnarray*}

\begin{remark}
Let $(\k^{\mathbb{Z}_2}{}^\tau\#_\sigma\k F, R)$ be a coquasitriangular Hopf algebra, $x, h\in G, f, f^\prime \in F.$
\begin{itemize}
\item[(1)]If $R(p_1\#1, p_g\#1)=\frac{1}{2}$, then either
$R(p_{xg}\#(g\triangleright f), p_h\# f^{\prime})R(p_g\#f, p_h\#1)=0,$
or $\tau(g,g;f)R(p_g\#(g\triangleright f), p_h\#1)R(p_g\#f, p_h\#1)=\frac{1}{4}.$
\item[(2)]If $R(p_1\#1, p_g\#1)=0$, the bilinear form $R$ satisfies:
\begin{eqnarray*}
&& \text{Either } R(p_x\# f, p_1\# f') = 0, \text{ or } R(p_1\# f, p_1\# 1) = R(p_x\# 1, p_1\# f') = 1; \\
&& \text{and either } R(p_x\# f, p_g\# f') = 0, \text{ or } R(p_1\# f, p_g\# 1) = R(p_x\# 1, p_1\# (g \triangleright f')) = 1.
\end{eqnarray*}
\end{itemize}
\end{remark}
Furthermore, if the group
$F$ is abelian, then several additional results can be drawn. Note that if $F=S$, Corollary \ref{lem:R(gS)} implies all the coquasitriangular structures $R$ are determined by bicharacters. We now turn to the case where $F\neq S$, that is, $T\neq \O.$
\begin{lemma}\label{lem:abelianRgh}
Let $F$ be an abelian group and consider the coquasitriangular Hopf algebra $(\k^{\mathbb{Z}_2}{}^\tau\#_\sigma\k F, R)$. If $F\neq S$, then
\begin{itemize}
\item [(1)]$R(p_g\# f, p_1\# f^\prime)=R(p_1\# f, p_g\# f^\prime)=R(p_g\# f, p_g\# f^\prime)=0$ for any $f, f^\prime \in S;$
\item [(2)]$R(p_1\# f, p_1\# f^\prime)R(p_g\# f, p_g\# f^{\prime\prime})=0$ for any $f^\prime\in F, f, f^{\prime\prime}\in T;$
\item [(3)]$R(p_1\# f, p_1\# f^\prime)R(p_g\# f^{\prime\prime}, p_g\# f^{\prime})=0$ for any $f\in F, f^\prime, f^{\prime\prime}\in T.$
\end{itemize}
\end{lemma}
\begin{proof}
For any $s \in S$ and $t \in T$, we have $st \in T$, and thus $s=stt^{-1}\in \{t_1t_2\mid t_1, t_2\in T\}$. Take
$h=l$ in (CQT1), we have
\begin{eqnarray*}
&&\sigma(h; f^\prime,  f^{\prime\prime})R(p_x\#f, p_h\#f^\prime f^{\prime\prime})\\&=&R(p_x\#f, p_h\# f^{\prime\prime})R(p_1\#f, p_h\#f^\prime)
+\tau(xg,g;f)R(p_{xg}\#(g\triangleright f), p_h\# f^{\prime\prime})R(p_g\#f, p_h\#f^\prime).
\end{eqnarray*}
Without loss of generality, suppose $f\in S$ and $s=f^\prime f^{\prime\prime}\in S,$ where $f^\prime, f^{\prime\prime}\in T$. It follows from Corollary \ref{coro:FabelianRgf=0} that $R(p_x\#f, p_h\#f^\prime f^{\prime\prime})=0$ for $h=g$ and $x\in \mathbb{Z}_2$. Moreover,
By substituting
$h\neq l$ into (CQT1), we find that
\begin{eqnarray*}
0=R(p_x\#f, p_{hg}\# f^{\prime\prime})R(p_1\#f, p_h\#f^\prime)
+\tau(xg,g;f)R(p_{xg}\#(g\triangleright f), p_{hg}\# f^{\prime\prime})R(p_g\#f, p_h\#f^\prime).
\end{eqnarray*}
Suppose $ f, f^{\prime\prime}\in T, h=1, x=g$.
Corollary \ref{coro:FabelianRgf=0} implies that $$R(p_1\# f, p_1\# f^\prime)R(p_g\# f, p_g\# f^{\prime\prime})=0.$$ Applying a similar analysis to (CQT2) gives the remaining results.
\end{proof}
Combining Lemmas \ref{lem:R11} and \ref{lem:abelianRgh}, we have the following corollary.
\begin{corollary}
Let $(\k^{\mathbb{Z}_2}{}^\tau\#_\sigma\k F, R)$ be a coquasitriangular Hopf algebra, where $F$ is an abelian group and $F\neq S$.
Then
 $R(p_1\#1, p_1\#1)=1,$ $R(p_1\#1, p_g\#1)=R(p_g\#1, p_1\#1)=R(p_g\#1, p_g\# 1)=0$.
 \end{corollary}
 We therefore establish the following proposition in the setting where
$F$ is an abelian group and $F\neq S$, which can be regarded as a generalization of \cite[Proposition 3.6]{ZL21}.
\begin{proposition}\label{prop:FabelianR}
Let $(\k^{\mathbb{Z}_2}{}^\tau\#_\sigma\k F, R)$ be a coquasitriangular Hopf algebra, where $F$ is an abelian group and $F\neq S$. Then one of the following two cases must hold:
\begin{itemize}
  \item [(1)]$R(p_1\#f, p_g\#f^\prime)=R(p_g\#f, p_1\#f^\prime)=R(p_g\#f, p_g\#f^\prime)=0$ for any $f, f^\prime \in F.$
  \item [(2)]$R(p_x\#f, p_y\# f^\prime)$ vanishes except in the following instances:
  \begin{itemize}
    \item [(i)]$R(p_1\#f, p_1\# f^\prime)\neq0$ only when $f, f^\prime \in S;$
    \item [(ii)]$R(p_1\#f, p_g\# f^\prime)\neq0$ only when $f \in T, f^\prime \in S;$
    \item [(iii)]$R(p_g\#f, p_1\# f^\prime)\neq0$ only when $f \in S, f^\prime\in T;$
    \item [(iv)]$R(p_g\#f, p_g\# f^\prime)\neq0$ only when $f, f^\prime \in T.$
  \end{itemize}
\end{itemize}
\end{proposition}
\begin{proof}
Suppose $R(p_g\#u, p_g\# v)=0$ for any $u, v \in T.$ According to Corollary \ref{coro:FabelianRgf=0} and Lemma \ref{lem:abelianRgh}(1), $R(p_g\#u, p_g\# v)=0$ for any $u, v \in F.$ For any $s\in S$, suppose $s=f^\prime f^{\prime\prime}$, where $f^\prime ,f^{\prime\prime}\in T$. Using (CQT1), we have
\begin{eqnarray*}
&&\sigma(g; f^\prime,  f^{\prime\prime})R(p_1\#f, p_g\#f^\prime f^{\prime\prime})\\&=&R(p_1\#f, p_g\# f^{\prime\prime})R(p_1\#f, p_g\#f^\prime)
+\tau(g,g;f)R(p_{g}\#(g\triangleright f), p_g\# f^{\prime\prime})R(p_g\#f, p_g\#f^\prime),
\end{eqnarray*}
where $f\in T.$ From Corollary \ref{coro:FabelianRgf=0}, $R(p_1\#f, p_g\#s)=0$ for any $f\in T, s\in S.$ It follows from Corollary \ref{coro:FabelianRgf=0} and Lemma \ref{lem:abelianRgh}(1) that $R(p_1\#u, p_g\#v)=0$ for any $u, v\in F.$ A similar argument shows that $R(p_g\#u, p_1\#v)=0$ for all $u, v\in F.$ Thus, the first case is established.

Suppose there exist some $t_1, t_2\in T$ such that $R(p_g\#t_1, p_g\# t_2)\neq0$. According to Lemma \ref{lem:abelianRgh}, $R(p_1\# t_1, p_1\# f)=R(p_1\# f, p_1\# t_2)=0$ for any $f\in F.$
Since $R$ is convolution-invertible, it follows that $$(R^{-1}*R)((p_1\#t)\otimes (p_1\# t_2))=(R^{-1}\otimes R)\Delta((p_1\#t)\otimes (p_1\# t_2))=1$$ for any $t\in T.$
Note that Corollary \ref{coro:FabelianRgf=0} immediately gives $R(p_1\# t, p_g\# t_2) = R(p_g\# t, p_1\# t_2) = 0$. One finds that $R(p_g\# t, p_g\# t_2) \neq 0$ for all $t \in T$. Moreover, Lemma \ref{lem:abelianRgh}(2) implies that $R(p_g\# t, p_g\# f) = 0$ for any $t \in T$ and $f \in F$, and a similar argument yields $R(p_g\# f, p_g\# t) = 0$ for any $f \in F$ and $t \in T$.
Combining Corollary \ref{coro:FabelianRgf=0} and Lemma \ref{lem:abelianRgh}(1), one can get the second case.
\end{proof}


\begin{thebibliography}{99}
\setlength{\itemsep}{0em}
\bibitem[Ago13]{Ago13}Agore, A. L. (2013).
Coquasitriangular structures for extensions of Hopf algebras. Applications. Glasg. Math. J. 55(1): 201-215. https://doi.org/10.1017/S0017089512000444
\bibitem[AD95]{AD95}Andruskiewitsch, N., Devoto, J. (1995). Extensions of Hopf algebras. Algebra. i Analiz 7(1): 22-61.
\bibitem[AEG01]{AEG01}Andruskiewitsch, N., Etingof, P., Gelaki, S. (2001). Triangular Hopf algebras with the Chevalley property. Machigan Math. J. 49(2): 277-298. DOI: 10.1307/mmj/1008719774
\bibitem[Dri86]{Dri86}Drinfeld, V. G. (1987).
Quantum groups. In: Proceedings of the International Congress of Mathematicians, Vol. 1, 2, 798-820.
American Mathematical Society, Providence, RI.
\bibitem[EGNO15]{EGNO15}Etingof, P., Gelaki, S., Nikshych, D., Ostrik, V. (2015). \textit{Tensor categories}. Providence, RI: American Mathematical Society.
\bibitem[Fuk97]{Fuk97}Fukuda, N. (1997).
Semisimple Hopf algebras of dimension 12.
Tsukuba J. Math. 21(1): 43-54. https://doi.org/10.21099/tkbjm/1496163160
\bibitem[GW98]{GW98}Gelaki, S., Westreich, S. (1998).
On the quasitriangularity of $U_q(sl_n)^\prime$.
J. London Math. Soc. (2) 57(1): 105-125. https://doi.org/10.1112/S0024610798005705
\bibitem[HL10]{HL10}Huang, H., Liu, G. (2010). On quiver-theoretic description for quasitriangularity of Hopf algebras. J. Algebra 323(10): 2848-2863. https://doi.org/10.1016/j.jalgebra.2010.02.041
\bibitem[JM09]{JM09}Jedwab, A., Montgomery, S. (2009). Representations of some Hopf algebras associated to the symmetric group $S_n$.
Algebr. Represent. Theory 12(1): 1-17. https://doi.org/10.1007/s10468-008-9099-0
\bibitem[Jia05]{Jia05}Jiao, Z. (2005).
The quasitriangular structures for a class of T-smash product Hopf algebras. Israel J. Math. 146: 125-147. https://doi.org/10.1007/BF02773530
\bibitem[JW05]{JW05}Jiao, Z., Wisbauer, R. (2005). The braided structures for $\omega$-smash coproduct Hopf algebras.
J. Algebra 287(2): 474-495. https://doi.org/10.1016/j.jalgebra.2005.02.025
\bibitem[Kei18]{Kei18}Keilberg, M. (2018). Quasitriangular structures of the double of a finite group.
Comm. Algebra 46(12): 5146-5178.
\bibitem[Lar71]{Lar71}Larson, R. G. (1971). Characters of Hopf algebras. J. Algebra 17: 352-368. https://doi.org/10.1016/0021-8693(71)90018-4
\bibitem[LT91]{LT91}Larson, R. G., Towber, J. (1991).
Two dual classes of bialgebras related to the concepts of ``quantum group'' and ``quantum Lie algebra''.
Comm. Algebra 19(12): 3295-3345.
\bibitem[Maj91]{Maj91}Majid, S. (1991).
Quantum groups and quantum probability. In: Quantum probability \& related topics, QP-PQ, VI, World Scientific Publishing Co., Inc., River Edge, NJ.
\bibitem[Mas02]{Mas02}Masuoka, A. (2002). Hopf algebra extensions and cohomology. In: New Directions in Hopf Algebras, in: Math. Sci. Res. Inst. Publ., vol. 3, Cambridge University Press, Cambridge, pp. 167-209.
\bibitem[Mon93]{Mon93}Montgomery, S. (1993). \textit{Hopf algebras and their actions on rings}. CBMS Reg. Conf. Ser. in Math. 82. Providence, RI: AMS.
\bibitem[Nat11]{Nat11}Natale, S. (2011).
On quasitriangular structures in Hopf algebras arising from exact group factorizations.
Comm. Algebra 39(12): 4763-4775. https://doi.org/10.1080/00927872.2011.617625
\bibitem[Nat06]{Nat06}Natale, S. (2006). $R$-matrices and Hopf algebra quotients. Int. Math. Res. Not. Art. ID 47182, 18 pp. https://doi.org/10.1155/IMRN/2006/47182
\bibitem[Nik19]{Nik19}Nikshych, D. (2019). Classifying braidings on fusion categories. In: Tensor categories and Hopf algebras,
Contemp. Math., 728, American Mathematical Society, RI, 155-167. https://doi.org/10.1090/conm/728/14660
\bibitem[Swe69]{Swe69}Sweedler, M. (1969). \textit{Hopf algebras}. New York: Benjamin.
\bibitem[Tak77]{Tak77}Takeuchi, M. (1977). Morita theorems for categories of comodules. J. Fac. Sci. Univ. Tokyo Sect. IA Math. 24(3): 629-644.
\bibitem[Tak81]{Tak81}Takeuchi, M. (1981). Matched pairs of groups and bismash products of Hopf algebras. Comm. Algebra 9(8): 841-882. https://doi.org/10.1080/00927878108822621
\bibitem[YL26]{YL26}Yu, J., Liu, G. (2026). Hopf algebras with the dual Chevalley property of discrete corepresentation type.
J. Algebra 688: 803-843.
\bibitem[YLZZ25]{YLZZ25}Yu, J., Liu, G., Zhou, K., Zhen, X.
A class of (infinite-dimensional) cosemisimple Hopf algebras constructed via abelian extensions. Preprint, arXiv:2506.04008
https://doi.org/10.48550/arXiv.2506.04008
\bibitem[ZZ24]{ZZ24}Zhang, H. Zhou, K. (2024).
Minimal triangular structures on abelian extensions.
Algebr. Represent. Theory 27(2): 1121-1136. https://doi.org/10.1007/s10468-023-10250-w
\bibitem[ZL21]{ZL21}Zhou, K, Liu, G. (2021).
On the quasitriangular structures of abelian extensions of $\Bbb{Z}_2$.
Comm. Algebra 49(11): 4755-4762. https://doi.org/10.1080/00927872.2021.1929274
\end{thebibliography}
\end{document}